\documentclass[11pt]{amsart}
\usepackage[T1]{fontenc}
\usepackage[utf8]{inputenc}
\usepackage{amsmath,amssymb,amsfonts}
\usepackage[english]{babel}

\usepackage{booktabs}

\usepackage{tikz, tikz-cd}
\usepackage{hyperref}

\numberwithin{equation}{section}

\newcommand{\ul}[1]{\underline{#1}}

\DeclareMathOperator{\Spec}{Spec}
\DeclareMathOperator*{\colim}{colim}

\DeclareMathOperator{\tr}{tr}

\DeclareMathOperator{\fib}{fib}
\DeclareMathOperator{\Frac}{Frac}
\DeclareMathOperator{\trdeg}{trdeg}
\DeclareMathOperator{\gr}{gr}
\DeclareMathOperator{\Fil}{Fil}
\DeclareMathOperator{\THH}{THH}
\DeclareMathOperator{\HH}{HH}
\DeclareMathOperator{\TR}{TR}
\DeclareMathOperator{\TC}{TC}
\DeclareMathOperator{\dlog}{dlog}

\newcommand{\red}{\mathrm{red}}

\newcommand{\SCR}{\mathrm{SCR}}
\newcommand{\op}{\mathrm{op}}

\newcommand{\Poly}{\mathrm{Poly}}

\newcommand{\Zar}{\mathrm{Zar}}

\newcommand{\N}{{\mathbb N}}

\newcommand{\bbL}{\mathbb{L}}

\newcommand{\Z}{{\mathbb Z}}
\newcommand{\F}{\mathbb{F}}

\newcommand{\p}{\mathfrak{p}}
\newcommand{\q}{\mathfrak{q}}
\newcommand{\fm}{\mathfrak{m}}

\newcommand{\sh}{\mathrm{sh}}

\DeclareMathOperator{\height}{ht}

\newcommand{\cO}{\mathcal{O}}
\newcommand{\cF}{\mathcal{F}}

\DeclareMathOperator{\pdim}{\mathit{p}-dim}

\newcounter{zaehler}
\theoremstyle{plain}

\newtheorem{introthm}[zaehler]{Theorem}
\newtheorem{introquestion}[zaehler]{Question}

\theoremstyle{plain}
\newtheorem{thm}{Theorem}[section]

\newtheorem{lem}[thm]{Lemma}
\newtheorem{lemma}[thm]{Lemma}
\newtheorem{cor}[thm]{Corollary}
\newtheorem{prop}[thm]{Proposition}

\newtheorem*{conj*}{Conjecture}

\theoremstyle{definition}
\newtheorem{ex}[thm]{Example}
\newtheorem{defn}[thm]{Definition}

\newtheorem{claim}[thm]{Claim}
\newtheorem{rmk}[thm]{Remark}

\title{Towards Vorst's conjecture in positive characteristic}
\author{Moritz Kerz}
\email{moritz.kerz@ur.de}
\address{Fakult\"at f\"ur Mathematik, Universit\"at Regensburg, 93040 Regensburg, Germany}

\author{Florian Strunk}
\email{florian.strunk@ur.de}
\address{Fakult\"at f\"ur Mathematik, Universit\"at Regensburg, 93040 Regensburg, Germany}

\author{Georg Tamme}
\email{georg.tamme@ur.de}
\address{Fakult\"at f\"ur Mathematik, Universit\"at Regensburg, 93040 Regensburg, Germany}
\thanks{The authors are supported by the DFG through CRC 1085 \textit{Higher Invariants} (Universit\"at Regensburg). }

\date{\today}

\begin{document}

\begin{abstract}
Vorst's conjecture relates the regularity of a ring with the $\mathbb{A}^1$-homotopy invariance of its $K$-theory.
We show a variant of this conjecture in positive characteristic.
\end{abstract}

\maketitle

\section{Introduction}

A commutative unital ring $A$ is called $K_n$-regular if the canonical map $K_n(A)\to
K_n(A[X_1,\ldots,X_m])$ is an isomorphism for all positive integers $m$.
By~\cite[Cor.~2.1]{Vorst} a $K_{n+1}$-regular ring is also $K_n$-regular. 
It is well known that a regular noetherian ring is $K_{n}$-regular  for all $n$.
In \cite{Vorst} Vorst conjectured the following partial converse.

\begin{conj*}[Vorst]
Let $k$ be a field, and let $A$ be essentially of finite type over $k$.
If $A$ is $K_{\dim(A)+1}$-regular, then $A$ is regular.
\end{conj*}

The case $\dim(A)=0$ is easy and the case $\dim(A)=1$ was shown by Vorst in \cite[Thm.~3.6]{Vorst}.
For fields $k$ of characteristic zero, Corti\~nas, Haesemeyer, and Weibel proved the conjecture in \cite[Thm.~0.1]{CHW-Vorst}.
Geisser and Hesselholt in~\cite{GH-vorst} proved the conjecture for $A$ of finite type
over a  perfect field $k$ of positive characteristic assuming resolution
of singularities. 

In order to formulate our results, we introduce the $p$-dimension of an $\F_p$-algebra $A$.
This number is defined as
\[
\pdim(A) = \sup \{  \pdim(k(\p)) + \height(\p) \, | \, \p \subset A \text{ prime ideal}\}
\]
where $\pdim(k(\p))$ is the $p$-rank of the residue field $k(\p)$, see Section~\ref{section:p-dim} for details.
In general $\pdim(A)\geq \dim(A)$ and equality holds for instance if $A$ is of finite type over a perfect field. 

\begin{introthm}\label{introthm:A}
Let $A$ be an excellent noetherian $\F_p$-algebra such that $[k(x):k(x)^p]<\infty$ for all points $x\in\Spec(A)$.
If $A$ is $K_{\pdim(A)+1}$-regular, then $A$ is regular.
\end{introthm}

In particular, this implies the result of Geisser and Hesselholt mentioned above without assuming resolution of singularities.
The theorem indicates that the condition that $A$ be essentially of finite type over a field is not necessary for the conjecture to hold.
In fact, using the result of Corti\~nas--Haesemeyer--Weibel, we show the following generalization in characteristic zero.

\begin{introthm}\label{introthm:B}
Let $A$ be an excellent noetherian ring of characteristic zero. If $A$ is
$K_{\dim(A)+1}$-regular, then $A$ is regular.
\end{introthm}

We also prove the following result for curves  in mixed characteristic. This seems to be
the first result of that form in mixed characteristic.

\begin{introthm}\label{introthm:mixed-char}
Let $A$ be an excellent noetherian ring with $\dim(A)\leq 1$ such that $A/\fm$ is perfect of characteristic $p> 2$ for every maximal ideal $\fm \subset A$.
If $A$ is $K_{2}$-regular, then $A$ is regular.
\end{introthm}

These results motivate the following question, already asked similarly by Vorst in~\cite{Vorst}.

\begin{introquestion}
  Let $A$ be an excellent noetherian ring which is $K_{\dim(A)+1}$-regular. Is $A$
  necessarily regular?
  \end{introquestion}

The proof of Theorem~\ref{introthm:mixed-char} is based on the $\dlog$-map to (absolute) de
Rham--Witt forms and  calculations of Hesselholt--Madsen~\cite{HesselholtMadsen-drW}.  
The proof of Theorem~\ref{introthm:A} essentially follows the strategy of
Geisser--Hessel\-holt. We replace resolution of singularities by an argument involving
the Zariski--Riemann space of $\Spec(A)$. This forces us to study $K$-theory of valuation
rings. In fact, we use the following vanishing result for the $K$-theory of valuation rings.

\begin{introthm}\label{introthm:C}
Let $V$ be a valuation ring of characteristic $p$ with field of fractions $F$.
Then $K_{i}(V; \Z/p) = 0$ for $i > \pdim(F)$.
\end{introthm}

Using a recent result of Clausen, Mathew, and Morrow \cite{ClausenMathewMorrow}, this follows from an analogous vanishing of topological cyclic homology (see Proposition~\ref{prop:tc-vanishing-on-valuation-rings}). 
The main new ingredient is a Cartier isomorphism for valuation rings, a proof of which was outlined in a letter of Gabber to the first author~\cite{Gabber_letter}.
We give a detailed account of his argument in the appendix.

This existence of a Cartier isomorphism is also used in a recent preprint of Kelly and Morrow \cite{KellyMorrow}, where they independently prove the finer result $K_i(V;\Z/p^r) \cong W_r\Omega^i_{V,\log}$ by similar methods.

\subsection*{Acknowledgement}
We thank Ofer Gabber for sending us an outline of the proof of the Cartier isomorphism
mentioned above and for helpful comments on our presentation of his results in the appendix.
We are grateful to the referees for a careful reading of our paper and several helpful comments. In particular, they suggested an alternative proof of the vanishing result for the topological cyclic homology of a valuation ring mentioned above.

\subsection*{Notation}
All rings in this text are assumed to be commutative and unital.
For a presheaf of spectra $E$ and an integer $p$ we write $E/p$ for the cofibre of the $p$-multiplication on $E$ and set
$E_i(X;\Z/p)=\pi_i(E/p(X))$.

\section{Preliminaries on the \textit{p}-dimension}
\label{section:p-dim}

Let $p$ be a prime number.
In this section we introduce the notion of $p$-dimension for $\F_{p}$-algebras. 
Let $R$ be an $\F_{p}$-algebra, and let $A$ be an $R$-algebra. By $R[A^{p}] \subset A$ we denote the $R$-subalgebra of $A$ generated by the $p$-th powers of the elements in $A$. Recall the following definition from \cite[$0_{\text{IV}}$(21.1.9)]{EGA}.

\begin{defn}\label{sec2.defpbase}
A family  $(x_{i})_{i\in I}$ of elements of $A$ is called a \emph{$p$-basis of $A$ over $R$} if the family of monomials
\begin{equation}
	\label{eq:p-monomials}
\prod_{i} x_{i}^{n_{i}} \quad (0 \leq n_{i} < p, n_{i} = 0 \text{ for all but finitely many } i \in I)
\end{equation}
forms a basis of $A$ as an $R[A^{p}]$-module. It is called \emph{$p$-independent over $R$} if the family \eqref{eq:p-monomials} is linearly independent over $R[A^{p}]$. The monomials \eqref{eq:p-monomials} are called the \emph{$p$-monomials} of the family $(x_{i})_{i\in I}$.
A $p$-basis for $A$ over $\F_{p}$ is called an \emph{absolute $p$-basis} or simply a \emph{$p$-basis for $A$}.
\end{defn}

We record the following simple observation.
\begin{lemma}
	\label{lem:p-basis-localization}
Assume that the family $(x_{i})_{i\in I}$ forms a $p$-basis of $A$ over $R$. If $S \subset A$ is a multiplicative set, then $(\frac{x_{i}}{1})_{i\in I}$ forms a $p$-basis of the localization $S^{-1}A$ over $R$.
\end{lemma}
\begin{proof}
To see that the family of $p$-monomials associated with $(\frac{x_{i}}{1})_{i \in I}$ generates $S^{-1}A$ as an $R[(S^{-1}A)^{p}]$-module, note that we can  write any element of $S^{-1}A$ in the form $\frac{a}{s^{p}}$ with $a\in A$ and $s\in S$. It is also easy to see that the family  $(\frac{x_{i}}{1})_{i \in I}$ is $p$-independent over $R$.
\end{proof}

If the family  $(x_{i})_{i\in I}$ is a $p$-basis of $A$ over $R$, then it is also a differential basis, i.e.~the family $(dx_{i})_{i\in I}$ is a basis of the  $A$-module of K\"ahler differentials $\Omega_{A/R}$, see \cite[$0_{\text{IV}}$ Cor.~21.2.5]{EGA}. The converse holds if $R \to A$ is a field extension \cite[Thm.~26.5]{Matsumura}. In particular, if $k' \subset k$ is any  extension of fields of characteristic $p$, then $k$ admits a $p$-basis over $k'$ and  any two $p$-bases have the same cardinality. This cardinality is called the \emph{$p$-rank of $k$ over $k'$}. The $p$-rank of $k$ over $\F_{p}$ is simply called the \emph{$p$-rank of k} and denoted by $\pdim(k)$. So $\pdim(k) = \dim_{k} \Omega_{k}$, where $\Omega_{(-)}= \Omega_{(-)/\F_{p}}$ denotes the module of absolute K\"ahler differentials.

\begin{lemma}
	\label{lem:p-dim-field-extension}
Assume that $k \subset k'$ is a finitely generated field extension in characteristic $p$.
Then 
\[
\pdim(k') = \pdim(k) + \trdeg_{k} k'.
\]
In particular, $\pdim(k) = \pdim(k')$ in case the extension is finite.
\end{lemma}
\begin{proof}
This follows from the exact sequence
\[
0 \to \Gamma_{k'/k/\F_{p}} \to \Omega_{k} \otimes_{k} k' \to \Omega_{k'} \to \Omega_{k'/k} \to 0,
\]
where the imperfection module $\Gamma_{k'/k/\F_{p}}$ is defined as the kernel of the map in the middle,
together with the Cartier equality \cite[Thm.~26.10]{Matsumura} 
\[
\dim_{k'} \Omega_{k'/k} - \dim_{k'} \Gamma_{k'/k/\F_{p}} = \trdeg_{k'} k,
\]
which holds since the field extension $k \subset k'$ is finitely generated. 
\end{proof}

\begin{defn}
	\label{def:p-dim}
For an $\F_p$-scheme $X$ the \emph{$p$-dimension} is defined as 
\[
\pdim(X) = \sup \{  \pdim(k(x)) + \dim(\cO_{X,x}) \, | \, x \in X \}
\]
where $k(x)$ is the residue field of $X$ at $x$. For an $\F_{p}$-algebra $A$ we set $\pdim(A) = \pdim(\Spec(A))$.
\end{defn}

Note that  $\dim( \cO_{\Spec(A), \p} ) = \height(\p)$ and 
\[
\pdim(A) = \sup \{  \pdim(A/\p') \, | \, \p' \subset A \text{ minimal prime ideal} \}.
\]

In the following, we collect some elementary properties of the $p$-dimension.

\begin{lemma}
	\label{lem:p-dim-for-p-basis}
If the noetherian $\F_{p}$-algebra $A$ is reduced and has an absolute $p$-basis of cardinality $r$, then $\pdim(A) = r$.
\end{lemma}
\begin{proof}
By Lemma~\ref{lem:p-basis-localization} it is enough to show that 
\[
\pdim(k) + \dim(A) = r 
\]
provided $A$ is moreover local with residue field $k$. Since $A$ is reduced, the Frobenius map $a \mapsto a^{p}$ is injective. According to \cite[$0_{\text{IV}}$ Thm.~21.2.7]{EGA} the $\F_{p}$-algebra $A$ is then formally smooth and hence regular by Theorem~22.5.8 there. Let $\fm$ denote the maximal ideal of $A$. The claim now follows from the fundamental exact sequence \cite[Thm.~25.2]{Matsumura}
\[
0 \to \fm/\fm^{2} \to \Omega_{A} \otimes_{A} k \to \Omega_{k} \to 0
\]
using $\dim_{k} \fm/\fm^{2} = \dim(A)$ by the regularity of $A$ and $\dim_{k}(  \Omega_{A} \otimes_{A} k ) = r$ by the remarks preceding Lemma~\ref{lem:p-dim-field-extension}.
\end{proof}

\begin{lemma}
	\label{lem:p-dim-finite-extension}
Let $A\subset B$ be a finite extension of integral  $\F_p$-algebras. Then $\pdim(A)=\pdim(B)$.
\end{lemma}
\begin{proof}
For a prime ideal $\q\subset B$ and $\p=\q\cap A\subset A$, we have the equality $\pdim(k(\q))=\pdim(k(\p))$ by Lemma~\ref{lem:p-dim-field-extension} as the residue field extension $k(\p)\subset k(\q)$ is finite.
Moreover, $\height(\q)\leq\height(\p)$, since $A \subset B$ is integral,  and hence $\pdim(B)\leq \pdim(A)$.
On the other hand, for every prime ideal $\p\subset A$ there exists a prime ideal $\q\subset B$ with $\p=\q\cap A$ and $\height(\q)\geq\height(\p)$ by going-up \cite[Thm.~9.4]{Matsumura}.
Therefore $\pdim(B)\geq \pdim(A)$.
This shows the claim.
\end{proof}

\begin{lemma}\label{lem:p-dim-and-dimension}
Let $A$ be either an $\F_p$-algebra of finite type over a field $k$ or a complete local noetherian $\F_{p}$-algebra with residue field $k$. Then
\[
\pdim(A) = \pdim(k) + \dim(A)
\]
where $\dim$ denotes the Krull dimension.
\end{lemma}
\begin{proof}
We may assume that $\pdim(k) < \infty$, since otherwise both sides of the equation are $\infty$.
By the remark after Definition~\ref{def:p-dim}, we may assume that $A$ is integral. Write $d$ for the Krull dimension of $A$.
In case $A$ is of finite type over a field $k$, Noether normalization \cite[(14.G)]{MatsumuraCA} yields a finite injective map $k[x_{1}, \dots, x_{d}] \to A$.
If $A$ is complete,  $A$  contains a coefficient field \cite[Thm.~28.3]{Matsumura}, and a choice of a system of parameters yields a finite injective map $k[[x_{1}, \dots, x_{d}]] \to A$.
By Lemma~\ref{lem:p-dim-finite-extension} it now suffices to show that 
\[
\pdim(k[x_{1}, \dots, x_{d}] ) = \pdim(k[[x_{1}, \dots, x_{d}]] ) = \pdim(k) + d.
\]
But in view of Lemma~\ref{lem:p-dim-for-p-basis} this follows from the fact that if $b_{1}, \dots, b_{r}$ form a $p$-basis of $k$, then $b_{1}, \dots, b_{r}, x_{1}, \dots, x_{d}$ form a $p$-basis of $k[x_{1}, \dots, x_{d}]$ and of $k[[x_{1}, \dots, x_{d}]]$, see \cite[Lemma~2.1.5]{GabberOrgogozo}.
\end{proof}

\begin{lemma}\label{lem:p-dim-and-transcendence-degree}
Let $A \subset B$ be an extension of finite type of integral noetherian $\F_{p}$-algebras. Then
$\pdim(B) \leq \pdim(A) + \trdeg_{\Frac(A)} \Frac(B)$.
\end{lemma}
\begin{proof}
Let $\q \subset B$ be a prime ideal, and let $\p = \q \cap A \subset A$. The dimension inequality \cite[Thm.~15.5]{Matsumura} gives
\[
\height(\q) + \trdeg_{k(\p)} k(\q) \leq \height(\p) + \trdeg_{\Frac(A)} \Frac(B).
\]
Since the field extension $k(\p) \subset k(\q)$ is finitely generated, Lemma~\ref{lem:p-dim-field-extension} implies that 
\[
\pdim(k(\q)) = \pdim(k(\p)) + \trdeg_{k(\p)} k(\q).
\]
Taken together, these facts imply the claim.
\end{proof}

\section{Derived differential forms and valuation rings}\label{sec:deriform}

Let $A$ be an $\F_{p}$-algebra.  Since the differential of the (absolute) de Rham complex $\Omega^{*}_{A}$ is Frobenius-linear, the subgroups of cycles
$Z\Omega^{i}_{A} \subset \Omega^{i}_{A}$ and boundaries $B\Omega^{i}_{A} \subset \Omega^{i}_{A}$, and the cohomology groups $H^{i}(\Omega^{*}_{A})$ are canonically $A$-modules via the Frobenius $a \mapsto a^{p}$.
There is a unique $A$-linear multiplicative map
\[
\gamma_{A} \colon \Omega^{i}_{A} \to H^{i}(\Omega^{*}_{A})
\]
characterized by $\gamma_{A}(1) = 1$ and $\gamma_{A}(dx) = x^{p-1}dx$, see \cite[Proof of Thm.~7.2]{Katz-nilpotent}. The map $\gamma_{A}$ is usually denoted by $C^{-1}$ and is called the \emph{inverse Cartier operator}. 
The classical theorem of Cartier \cite[Thm.~7.2]{Katz-nilpotent} says that $\gamma_{A}$ is
an isomorphism provided that $A$ is smooth over a perfect field. If $\gamma_{A}$ is an isomorphism,  the inverse induces a map $C\colon Z\Omega^{i}_{A} \to \Omega^{i}_{A}$ called the \emph{Cartier operator}.
In the appendix (Corollary \ref{cor:Cartier-iso-valuation-ring}(iii)) we explain a proof
of the following
result due to Gabber~\cite{Gabber_letter}.
\begin{thm}
	\label{thm:Cartier-iso-valuation-ring}
Let $V$ be a valuation ring of characteristic $p$. Then the inverse Cartier operator $\gamma_{V}$ is an isomorphism.
\end{thm}

In the following, we need some nonabelian derived functors, see \cite[\S5.5.8]{htt} and \cite[Ch.~25]{sag} for a  general treatment. 
We denote by $\Poly_{\F_{p}}$  the category of polynomial $\F_{p}$-algebras in finitely many variables, and by $\SCR_{\F_{p}}$  the $\infty$-category  obtained from the category of simplicial commutative $\F_{p}$-algebras by inverting the quasi-isomorphisms. Then $\Poly_{\F_{p}}$ is a full subcategory of $\SCR_{\F_{p}}$.  Moreover, if $D$ is any $\infty$-category that admits sifted colimits, then any functor $F\colon \Poly_{\F_{p}} \to D$  admits an essentially unique extension $LF\colon \SCR_{\F_{p}} \to D$ which preserves sifted colimits. The functor $LF$ is called the derived functor of $F$.

If $F$ is a functor $\SCR_{\F_{p}} \to D$, we still denote by $LF$ the derived functor of the restriction of $F$ to $\Poly_{\F_{p}}$. There is a natural transformation $LF \to F$. If the functor $F$ commutes with filtered colimits, then $LF(A) \simeq \colim_{\Delta^{\op}} F(P)$ where $P \xrightarrow{\sim} A$ is a simplicial resolution by free $\F_{p}$-algebras.

\begin{ex}
For $F = \Omega_{(-)}$, viewed as functor on $\Poly_{\F_{p}}$ with values in the derived  $\infty$-category $D(\F_{p})$ of  $H\F_{p}$-modules, 
we obtain the functor $L\Omega_{(-)}\colon \SCR_{\F_{p}} \to D(\F_{p})$. For any $\F_{p}$-algebra $A$, $L\Omega_{A}$ is equivalent to the underlying $H\F_{p}$-module of the cotangent complex $\bbL_{A/\F_{p}} \in D(A)$, see \cite[\S25.3]{sag}.

Moreover, $L\Omega^{i}_{A}$ is equivalent to the underlying $H\F_{p}$-module of the derived exterior power $L\bigwedge^{i}_{A}\bbL_{A/\F_{p}}$, see \cite[\S25.2]{sag}. This follows directly from the constructions and the fact that for a polynomial $\F_{p}$-algebra $P$ the $P$-module $\Omega_{P}$ is free.
\end{ex}

The following result is essentially due to Gabber and Ramero.
\begin{thm}
	\label{thm:Gabber-Ramero-Omega-i}
Let $V$ be a valuation ring of characteristic $p$, and let $i \geq 0$. Then the following hold.
\begin{enumerate} 
\item $L\Omega^{i}_{V} \simeq \Omega^{i}_{V}[0]$.
\item The $V$-module $\Omega^{i}_{V}$ is torsion-free or, equivalently, $\Omega^{i}_{V}$  flat.
\end{enumerate}
\end{thm}

\begin{proof}
  For $i=0$ both claims are clear; for $i=1$, assertion (1) follows from
  \cite[Thm.~6.5.12]{GabberRamero} and (2) is Corollary 6.5.21 there. An alternative
  argument due to Gabber is explained in the appendix,
  see Corollary~\ref{cor:Cartier-iso-valuation-ring}. Let now $i \geq 2$. Since
  $\bbL_{V} \simeq \Omega_{V}[0]$ and $\Omega_{V}$ is torsion-free and hence flat, it
  follows that
\[
L\Omega^{i}_{V} \simeq L\bigwedge\nolimits^{i}_{V}\bbL_{V} \simeq \bigwedge\nolimits^{i}_{V}\Omega_{V} = \Omega^{i}_{V}[0],
\]
see \cite[Prop.~25.2.3.4]{sag}.
It remains to prove (2) for $i \geq 2$. As $\Omega_{V}$ is torsion free, it is isomorphic
to a filtered colimit of finitely generated torsion free modules, which are  free by~\cite[VI.3.6~Lemma~1]{Bourbaki_comalg}. Since exterior powers of free modules are free, and since taking exterior powers commutes with filtered colimits, $\Omega^{i}_{V}$ is a filtered colimit of free modules and hence flat.
\end{proof}

We want to prove the analog of Theorem~\ref{thm:Gabber-Ramero-Omega-i} for the de Rham--Witt groups. For this we need some preparations. 
\begin{lemma}
	\label{lem:derived-BOmega}
Let $V$ be a valuation ring of characteristic $p$. Then $LB\Omega^{i}_{V} \simeq B\Omega^{i}_{V}[0]$ and $LZ\Omega^{i}_{V} \simeq Z\Omega^{i}_{V}[0]$.
\end{lemma}
\begin{proof}
By the Cartier isomorphism, recalled at the beginning of this section, we have the following short exact sequence of functors on $\Poly_{\F_{p}}$:
\[
0 \to B\Omega^{i} \to Z\Omega^{i} \xrightarrow{C} \Omega^{i} \to 0
\]
Taking derived functors and evaluating at $V$ we obtain the cofibre sequence
\begin{equation*}
\begin{tikzcd}
LB\Omega^{i}_{V} \arrow[r] & LZ\Omega^{i}_{V} \arrow[r] & L\Omega^{i}_{V}.
\end{tikzcd}
\end{equation*}
As by Theorem~\ref{thm:Cartier-iso-valuation-ring} the inverse Cartier operator $\gamma_{V}$ is also an isomorphism for the valuation ring $V$, we also have  a cofibre sequence
\[
\begin{tikzcd}
B\Omega^{i}_{V}[0] \arrow[r]& Z\Omega^{i}_{V}[0] \arrow[r, "C"] & \Omega^{i}_{V}[0].
\end{tikzcd}
\]
As the inverse Cartier operator is natural, the following diagram, in which the vertical maps are the canonical ones, commutes:
\begin{equation*}
\begin{tikzcd}
LB\Omega^{i}_{V} \arrow[r]\arrow[d] & LZ\Omega^{i}_{V} \arrow[r]\arrow[d] & L\Omega^{i}_{V} \arrow[d, "\simeq"]\\ 
B\Omega^{i}_{V}[0] \arrow[r]& Z\Omega^{i}_{V}[0] \arrow[r] & \Omega^{i}_{V}[0].
\end{tikzcd}
\end{equation*}
The right vertical map in this diagram is an equivalence by Theorem~\ref{thm:Gabber-Ramero-Omega-i}. So we see that if we prove the assertion about $B\Omega^{i}$, then also the assertion about $Z\Omega^{i}$ follows. We now argue by induction on $i$. As $B\Omega^{0} = 0$, the assertion is clearly true in the case $i = 0$.
Similarly as above, the exact sequence of functors
\[
0 \to Z\Omega^{i} \to \Omega^{i} \xrightarrow{d} B\Omega^{i+1} \to 0
\]
gives rise to the following diagram of cofibre sequences
\[
\begin{tikzcd}
LZ\Omega^{i}_{V} \arrow[r]\arrow[d] & L\Omega^{i}_{V} \arrow[r]\arrow[d, "\simeq"] & LB\Omega^{i+1}_{V} \arrow[d] \\
Z\Omega^{i}_{V}[0] \arrow[r] & \Omega^{i}_{V}[0] \arrow[r] & B\Omega^{i+1}_{V}[0].
\end{tikzcd}
\]
By induction, the left vertical map is an equivalence, hence so is the right vertical map.
\end{proof}

We next recall the definition of the higher boundaries and cycles in the de Rham complex from \cite[I.2.2]{Illusie-drW}. Let  $A$ be an  $\F_{p}$-algebra for which the inverse Cartier operator $\gamma_{A}\colon \Omega^{i}_{A} \to H^{i}(\Omega^{*}_{A})$ is an isomorphism, for example a polynomial algebra or a valuation ring.
One defines the chain of subgroups
\begin{multline*}
0 = B_{0}\Omega^{i}_{A} \subset B_{1}\Omega^{i}_{A} \subset \cdots \subset B_{n}\Omega^{i}_{A} \subset \cdots \\
	\cdots \subset Z_{n}\Omega^{i}_{A} \subset \cdots \subset Z_{1}\Omega^{i}_{A} \subset Z_{0}\Omega^{i}_{A} = \Omega^{i}_{A}
\end{multline*}
inductively by setting $B_{1}\Omega^{i}_{A} = B\Omega^{i}_{A}$, $Z_{1}\Omega^{i}_{A} = Z\Omega^{i}_{A}$ and requiring that $\gamma$ induces isomorphisms
\begin{equation}
	\label{eq:def:higher-cycles-and-boundaries}
B_{n}\Omega^{i}_{A} \xrightarrow{\simeq} B_{n+1}\Omega^{i}_{A}/B\Omega^{i}_{A} \quad  \text{ and } \quad
Z_{n}\Omega^{i}_{A} \xrightarrow{\simeq} Z_{n+1}\Omega^{i}_{A}/B\Omega^{i}_{A}. 
\end{equation}
Since the $B_{n}\Omega^{i}$ and $Z_{n}\Omega^{i}$ define functors on $\Poly_{\F_{p}}$, we get
 derived functors $LB_{n}\Omega^{i}$ and $LZ_{n}\Omega^{i}$ defined on $\SCR_{\F_{p}}$. In particular, we may evaluate them on any $\F_{p}$-algebra.

\begin{lemma}
	\label{lem:derived-higher-boundaries}
Let $V$ be a valuation ring of characteristic $p$. Then $LB_{n}\Omega^{i}_{V} \simeq B_{n}\Omega^{i}_{V}[0]$ and $LZ_{n}\Omega^{i}_{V} \simeq Z_{n}\Omega^{i}_{V}[0]$.
\end{lemma}
\begin{proof}
We argue by induction on $n$. The case $n=1$ is done in Lemma~\ref{lem:derived-BOmega}. By definition, we have the following exact sequence of functors on $\Poly_{\F_{p}}$:
\[
0 \to B\Omega^{i} \to B_{n+1}\Omega^{i} \xrightarrow{C} B_{n}\Omega^{i} \to 0
\]
Taking derived functors and evaluating at $V$ gives the inductive step, similarly as in the proof of Lemma~\ref{lem:derived-BOmega}. The proof for $Z_{n}\Omega^{i}$ is the same.
\end{proof}

Next we recall that to any $\F_{p}$-algebra $A$ one functorially associates its de Rham--Witt pro-complex $\{W_{n}\Omega^{*}_{A}\}_{n}$, see \cite{Illusie-drW}, \cite[Thm.~A]{HesselholtMadsen-drW}. The structure maps of the pro-system are denoted by $R$ and are called restriction maps. It follows directly from \cite[Thm.~I.1.3]{Illusie-drW} that the restriction maps are surjective.
We view $W_{n}\Omega^{i}$  as a functor on $\F_{p}$-algebras with values in $D(\Z)$. So we have its derived functor available.

\begin{prop}
	\label{prop:Witt-analog-of-Gabber-Ramero}
Let $V$ be a valuation ring of characteristic $p$ with field of fractions $F$, and let $n\geq 1$ and $i \geq 0$. Then the following hold.
	\begin{enumerate}
	\item $LW_{n}\Omega^{i}_{V} \simeq W_{n}\Omega^{i}_{V}[0]$
	\item The natural map $W_{n}\Omega^{i}_{V} \to W_{n}\Omega^{i}_{F}$ is injective.
	\end{enumerate} 
\end{prop}
\begin{proof}
(1) As a first step, we  treat the case $i=0$. Note that $W_{n}\Omega^{0}$ is the ring $W_{n}$ of Witt vectors of length $n$. We claim that in fact $LW_{n}(A) \simeq W_{n}(A)[0]$ for any $\F_{p}$-algebra $A$. This is clear for $n=1$. For $n >1$ the claim follows by induction using the short exact sequence 
\[
0 \to A \xrightarrow{V^{n}} W_{n+1}(A) \to W_{n}(A) \to 0
\]
which is natural in $A$.

We next prove that $\pi_{0}LW_{n}\Omega^{i}_{A} = W_{n}\Omega^{i}_{A}$ for any $\F_{p}$-algebra $A$.
Let $P \xrightarrow{\sim} A$ be a simplicial resolution of $A$ by free $\F_{p}$-algebras. 
We have to  show that the sequence
\[
 W_{n}\Omega^{i}_{P_{1}} \xrightarrow{\partial_{0} - \partial_{1}} W_{n}\Omega^{i}_{P_{0}} \to W_{n}\Omega^{i}_{A} \to 0
 \]
is exact. According to \cite[Lemma~2.4]{GeisserHesselholtComplete}, the  right-hand map is surjective and its kernel is generated by elements of the form $x\cdot \omega$ and $dx\cdot\omega$ with $x \in \ker( W_{n}(P_{0}) \to W_{n}(A))$ and $\omega \in W_{n}\Omega^{i}_{P_{0}}$ or $\omega\in W_{n}\Omega^{i-1}_{P_{0}}$, respectively. By the first step of the proof we know that $W_{n}(P) \xrightarrow{\sim} W_{n}(A)$ is a simplicial resolution. In particular, for $x$ as above there exists an element $y \in W_{n}(P_{1})$ such that $\partial_{0}(y) = x$, $\partial_{1}(y) = 0$. If $s$ denotes the degeneracy map $W_{n}(P_{0}) \to W_{n}(P_{1})$, then $y\cdot s(\omega)$ respectively $dy\cdot s(\omega)$ is a preimage of $x\cdot\omega$ respectively $dx\cdot\omega$ under $\partial_{0} - \partial_{1}$, thus showing the desired exactness.

It remains to prove that $LW_{n}\Omega^{i}_{V}$ is concentrated in degree 0 for $i \geq 1$ and $n \geq 1$.
We argue by induction on $n$. Since $W_{1}\Omega^{i} \cong \Omega^{i}$ (see \cite[Thm.~I.1.3]{Illusie-drW}), the case $n=1$ follows from Theorem~\ref{thm:Gabber-Ramero-Omega-i}(1).
Assume our assertion  is proven for some $n\geq 1$. 
Recall from \cite[I.3.1]{Illusie-drW} the canonical filtration on $W\Omega^{i}_{A} = \lim_{n} W_{n}\Omega^{i}_{A}$ given by 
\[
\Fil^{n}W\Omega^{i}_{A} = \ker ( W\Omega^{i}_{A} \xrightarrow{\text{can.}} W_{n}\Omega^{i}_{A} )
\]
 for any smooth $\F_{p}$-algebra $A$. Its associated graded pieces sit in a short exact sequence
\begin{equation}
	\label{seq:gr-drW}
	0 \to \gr^{n}W\Omega^{i}_{A} \to W_{n+1}\Omega^{i}_{A} \xrightarrow{R} W_{n}\Omega^{i}_{A} \to 0.
\end{equation}
Viewed  as a short exact sequence of functors in the  smooth $\F_{p}$-algebra $A$, the latter 
gives rise to a cofibre sequence of derived functors 
\begin{equation}
	\label{seq:cofibre-sequence-1}
L\gr^{n}W\Omega^{i} \to LW_{n+1}\Omega^{i} \to LW_{n}\Omega^{i}.
\end{equation}
By induction, it now suffices to show that $L\gr^{n}W\Omega^{i}_{V}$ is concentrated in degree 0.
By a fundamental result of Illusie  \cite[Cor.~I.3.9]{Illusie-drW}, there is a natural short exact sequence
\begin{equation}
	\label{seq:Illusie-fundamental-sequence}
	0 \to \Omega^{i}_{A} / B_{n}\Omega^{i}_{A} \to \gr^{n}W\Omega^{i}_{A} \to \Omega^{i-1}_{A} / Z_{n}\Omega^{i-1}_{A} \to 0
\end{equation}
for any smooth $\F_{p}$-algebra $A$. By the same argument as before, we now finish the proof of (1) by noting that $L(\Omega^{i}/B_{n}\Omega^{i})_{V}$ and $L(\Omega^{i-1}/Z_{n}\Omega^{i-1})_{V}$ are concentrated in degree 0. Indeed, this follows immediately from Lemma~\ref{lem:derived-higher-boundaries}  together with Theorem~\ref{thm:Gabber-Ramero-Omega-i}(1).

\medskip

\noindent (2) We again argue by induction on $n$. The case $n=1$ is Theorem~\ref{thm:Gabber-Ramero-Omega-i}(2). Note that by part (1) applied to the trivial valuation ring  $F$, we also have $LW_{n}\Omega^{i}_{F} \simeq W_{n}\Omega^{i}_{F}$ for all $n \geq 1$ and $i \geq 0$ and Lemma~\ref{lem:derived-higher-boundaries} holds with $V$ replaced by $F$. Using the cofibre sequence \eqref{seq:cofibre-sequence-1} and the result of part (1), we see that for the inductive step it suffices to show that 
\[
\pi_{0} L\gr^{n}W\Omega^{i}_{V} \to \pi_{0} L\gr^{n}W\Omega^{i}_{F}
\]
is injective for all $n \geq 0$. Then, using the cofibre sequence of derived functors obtained from \eqref{seq:Illusie-fundamental-sequence}, we  reduce to proving that the maps
\[
\Omega^{i}_{V}/B_{n}\Omega^{i}_{V} \to \Omega^{i}_{F}/B_{n}\Omega^{i}_{F} \quad \text{ and } \quad \Omega^{i}_{V}/Z_{n}\Omega^{i}_{V} \to \Omega^{i}_{F}/Z_{n}\Omega^{i}_{F}
\]
are injective for any $i \geq 0$ and $n \geq 1$.

We now prove the latter statement for the higher cycles by induction on $n$. To simplify notation, we drop the index $V$ or $F$, whenever a statement holds for both of them. In the  case $n=1$, the desired injectivity follows from the injectivity of $d\colon \Omega^{i}/Z\Omega^{i} \to \Omega^{i+1}$ and of $\Omega^{i+1}_{V} \to \Omega^{i+1}_{F}$, see Theorem~\ref{thm:Gabber-Ramero-Omega-i}(2). Assume, we have proven injectivity for some $n$. It follows from the definition of higher cycles \eqref{eq:def:higher-cycles-and-boundaries} that the inverse Cartier operator $\gamma$ induces an isomorphism 
\[
\Omega^{i}/Z_{n}\Omega^{i} \xrightarrow{\cong} Z\Omega^{i}/Z_{n+1}\Omega^{i}.
\]
Combining this isomorphism with the exact sequence
\[
0 \to Z\Omega^{i}/Z_{n+1}\Omega^{i} \to \Omega^{i}/Z_{n+1}\Omega^{i} \to \Omega^{i}/Z\Omega^{i} \to 0
\]
and the injectivity for $n=1$ gives the inductive step.
The proof for higher boundaries is  the same, replacing in the above formulas $Z_{n}$ by $B_{n}$ and $Z_{n+1}$ by $B_{n+1}$, and using the injection $\Omega^{i}/B\Omega^{i} \hookrightarrow \Omega^{i}$ given by the Cartier operator for the case $n=1$.
\end{proof}

\begin{cor}
	\label{cor:dRW-p-torsion-free}
The $p$-multiplication $p\colon W_{n}\Omega^{i}_{V} \to W_{n}\Omega^{i}_{V}$ factors as 
\[
W_{n}\Omega^{i}_{V} \xrightarrow{R} W_{n-1}\Omega^{i}_{V} \xrightarrow{\ul{p}} W_{n}\Omega^{i}_{V}
\]
and the induced map $\ul{p}$ is injective. In particular, the pro-group $\{ W_{n}\Omega^{i}_{V}\}_{n}$ is $p$-torsion-free.
\end{cor}
\begin{proof}
For $V$ replaced by any smooth $\F_{p}$-algebra the same assertion is proved in \cite[Prop.~I.3.4]{Illusie-drW} and remains true for ind-smooth $\F_{p}$-algebras.
Hence the $p$-multiplication on the derived functor $LW_{n}\Omega^{i}$ factors as 
\[
LW_{n}\Omega^{i} \xrightarrow{R} LW_{n-1}\Omega^{i} \xrightarrow{\ul{p}} LW_{n}\Omega^{i}
\]
Evaluating on $V$ and using Proposition~\ref{prop:Witt-analog-of-Gabber-Ramero}(1) gives the desired factorisation.
Using part (2) of the proposition and Illusie's result for the ind-smooth $\F_{p}$-algebra $F$ proves the asserted injectivity.
\end{proof}

We will use the above results to prove a vanishing result for topological cyclic homology of valuation rings in characteristic $p$.
Recall that for any (simplicial) $\F_{p}$-algebra $A$ one defines the spectra
\[
\TR^{n}(A;p) = \THH(A)^{C_{p^{n-1}}}
\]
as the genuine fixed points of the  topological Hochschild homology spectrum $\THH(A)$.
There are natural  maps $R, F \colon \TR^{n}(A;p) \to \TR^{n-1}(A;p)$ called restriction and Frobenius, and one defines the spectrum $\TR(A;p)  = \lim_{R} \TR^{n}(A;p)$. The topological cyclic homology of $A$ then sits in a fibre sequence
\begin{equation}
	\label{seq:def-TC}
\TC(A) \to \TR(A;p) \xrightarrow{1-F} \TR(A;p).
\end{equation}
Hesselholt--Madsen prove in \cite[Prop.~5.4]{HesselholtMadsenWitt} that $\TR_{*}^{n}(\F_{p})$ is isomorphic to the polynomial ring $\Z/p^{n}[\sigma_{n}]$ with $\sigma_{n}$ of degree 2 and the restriction map sends $\sigma_{n}$ to $p\sigma_{n-1}$ up to a unit in $\Z/p^{n-1}$.
Hesselholt shows in \cite{HesselholtTypicalCurves} that for  any $\F_{p}$-algebra $A$ and every $n\geq 1$ one gets a naturally induced map of graded rings
\begin{equation}
	\label{eq:canonical-map-drW-TR}
W_{n}\Omega^{*}_{A} [\sigma_{n}] \to  \TR_{*}^{n}(A;p)
\end{equation}
and that \eqref{eq:canonical-map-drW-TR} is an isomorphism provided that $A$ is smooth over $\F_{p}$, see \cite[Thm.~B]{HesselholtTypicalCurves}. Since both sides of \eqref{eq:canonical-map-drW-TR} commute with filtered colimits, \eqref{eq:canonical-map-drW-TR} is an isomorphism if $A$ is only assumed to be ind-smooth over $\F_{p}$.

\begin{prop}
	\label{lem:dRW-TR-iso}
Let $V$ be a valuation ring of characteristic $p$. Then the map 
\[
W_{n}\Omega^{*}_{V} [\sigma_{n}] \to  \TR_{*}^{n}(V;p)
\]
from \eqref{eq:canonical-map-drW-TR} is an isomorphism of graded rings. In particular, the natural map of pro-groups 
$\{W_{n}\Omega^{i}_{V} \}_{n} \to \{ \TR^{n}_{i}(V;p)\}_{n}$ is an isomorphism.
\end{prop}
\begin{proof}
Let $P \xrightarrow{\sim} V$ be a simplicial resolution by free $\F_{p}$-algebras. As the spectrum valued functor $\TR^{n}(-;p)$ on  $\SCR_{\F_{p}}$ commutes with sifted colimits (this follows inductively from the basic cofibre sequence \cite[(1.3.10)]{HesselholtTypicalCurves}), we have an equivalence
\(
\colim_{\Delta^{\mathrm{op}}} \TR^{n}(P;p) \simeq \TR^{n}(V;p)
\)
and hence a convergent spectral sequence 
\[
E^{1}_{rs} = \TR^{n}_{s}(P_{r};p) \Longrightarrow \TR^{n}_{r+s}(V;p).
\]
By Hesselholt's result above,  we have  $W_{n}\Omega^{s}_{P_{r}}[\sigma_{n}] \cong E^{1}_{rs}$ via the canonical map, and hence 
$\pi_{r}(LW_{n}\Omega^{s}_{V}[\sigma_{n}]) \cong E^{2}_{rs}$ by the definition of derived functors.
So Proposition~\ref{prop:Witt-analog-of-Gabber-Ramero} implies $E^{2}_{rs} = 0$ for $r > 0$ and $W_{n}\Omega^{s}_{V}[\sigma_{n}] \cong E^{2}_{0s}$. This gives the first claim. The second statement follows immediately from the fact that the transition map sends $\sigma_{n}$ to $p\sigma_{n-1}$ up to a unit.
\end{proof}

\begin{rmk}
With a  similar approach, the result of Proposition~\ref{lem:dRW-TR-iso} has been shown by Kelly and Morrow \cite[\S 2.3]{KellyMorrow} for Cartier-smooth $\F_{p}$-algebras.
\end{rmk}

\begin{prop}\label{prop:tc-vanishing-on-valuation-rings}
Let $V$ be a valuation ring of characteristic $p$ with field of fractions $F$. Then $TC_{i}(V; \Z/p) = 0$ for $i > \pdim(F)$.
\end{prop}
\begin{proof}
In view of the fibre sequence \eqref{seq:def-TC} and the Milnor sequence for $\TR$, it suffices to show that the pro-group $\{ \TR^{n}_{i}(V;p,\Z/p) \}_{n}$ vanishes for $i > \pdim(F)$. From Corollary~\ref{cor:dRW-p-torsion-free} and Proposition~\ref{lem:dRW-TR-iso} we deduce an isomorphism $\{ W_{n}\Omega^{i}_{V} / p \}_{n} \cong \{ \TR^{n}_{i}(V;p,\Z/p) \}_{n}$. Since the map $W_{n}\Omega^{i}_{V} \to W_{n}\Omega^{i}_{F}$ is injective  by Proposition~\ref{prop:Witt-analog-of-Gabber-Ramero}, it now suffices to observe that $W_{n}\Omega^{i}_{F} = 0$ for $i>\pdim(F)$. Indeed, since $F$ is ind-smooth over $\F_{p}$ and $\Omega^{i}_{F}$ vanishes for $i>\pdim(F)$ this follows easily by induction using the exact sequences~\eqref{seq:gr-drW} and \eqref{seq:Illusie-fundamental-sequence} (note that $\Omega^{i}_{F} = Z\Omega^{i}_{F} = \dots = Z_{n}\Omega^{i}_{F}$ for $i \geq \pdim(F)$).
\end{proof}

Here is an alternative proof of Proposition~\ref{prop:tc-vanishing-on-valuation-rings} indicated to us by a referee. It does not use the computation of $\TR$ of a valuation ring in Proposition~\ref{lem:dRW-TR-iso} and thus avoids the use of the derived de Rham--Witt complex.
\begin{proof}
Let $A$ be a polynomial $\F_{p}$-algebra. We define a complete exhaustive decreasing $\Z$-indexed filtration on $\TC(A;\Z/p) = \TC(A)/p$ via
\[
\Fil^{n}\TC(A; \Z/p) = \fib( \tau_{\geq n}(\TR(A;p)/p) \xrightarrow{1-F} \tau_{\geq n}(\TR(A;p)/p) ).
\]
It follows from Hesselholt's Hochschild--Kostant--Rosenberg theorem \cite[Thm.~B]{HesselholtTypicalCurves} that there is a natural isomorphism $\TR_{*}(A;p)\cong W\Omega^{*}_{A}$. Since $W\Omega^{*}_{A}$ is $p$-torsion free \cite[Prop.~I.3.4]{Illusie-drW}, it follows that $\pi_{*}( \TR(A;p)/p) \cong W\Omega^{*}_{A}/p$. For the associated graded of the above filtration on $\TC$ we thus obtain
\[
\gr^{n}\TC(A;\Z/p) \simeq \fib( W\Omega^{n}_{A}/p \xrightarrow{1-F} W\Omega^{n}_{A}/p )[n].
\]
The square of abelian groups
\[
\begin{tikzcd}
 W\Omega^{n}_{A}/p \ar[d]\ar[r, "1-F"] & W\Omega^{n}_{A}/p \ar[d] \\ 
 \Omega^{n}_{A} \ar[r, "1-C^{-1}"] & \Omega^{n}_{A}/B\Omega^{n}_{A} 
\end{tikzcd}
\]
where the vertical maps are the canonical projections is bicartesian, see the proof of \cite[Prop.~2.26]{ClausenMathewMorrow}, and thus cartesian when viewed as a diagram of spectra. We thus obtain a natural equivalence
\[
\gr^{n}\TC(A;\Z/p) \simeq \fib( \Omega^{n}_{A} \xrightarrow{1-C^{-1}} \Omega^{n}_{A}/B\Omega^{n}_{A})[n].
\]
As $\TC(-)/p$ commutes with sifted colimits \cite[Cor.~2.15]{ClausenMathewMorrow}, we can derive the above filtration on $\TC(-)/p$ of polynomial $\F_{p}$-algebras and obtain a filtration on $\TC(-)/p$ of an arbitrary $\F_{p}$-algebra $R$. Notice that this filtration is still complete (as $\Fil^{n}\TC(A;\Z/p)$ is $n-1$-connective) and exhaustive. By the above, its graded pieces are given by
\[
\gr^{n}\TC(R;\Z/p) \simeq \fib( L\Omega^{n}_{R} \xrightarrow{1-C^{-1}} L(\Omega^{n}/B\Omega^{n})_{R} ) [n],
\]
where as usual $L$ denotes nonabelian derived functors. 

Now let $V$ be a valuation ring of characteristic $p$ with field of fractions $F$. It follows from Lemma~\ref{lem:derived-BOmega} that $L\Omega^{n}_{V} \simeq \Omega^{n}_{V}[0]$ and also that $L(\Omega^{n}/B\Omega^{n})_{V} = \Omega^{n}_{V}/B\Omega^{n}_{V}[0]$ (this is where the Cartier isomorphism for valuation rings is used).
Thus
\[
\gr^{n}\TC(V;\Z/p) \simeq \fib( \Omega^{n}_{V} \xrightarrow{1-C^{-1}} \Omega^{n}_{V}/B\Omega^{n}_{V} ) [n]
\]
is concentrated in degrees $n-1, n$. Furthermore, from Theorem~\ref{thm:Gabber-Ramero-Omega-i} we deduce that $\Omega^{n}_{V}$ vanishes for $n>\pdim(F)$ and thus 
\[
\gr^{n}\TC(V;\Z/p) = 0
\]
for $n>\pdim(F)$. Since the filtration is complete, i.e. $\lim_{n} \Fil^{n}\TC(V;\Z/p) = 0$, this implies that $\Fil^{n}\TC(V;\Z/p) = 0$ for $n>\pdim(F)$. Inductively, we then get that $\Fil^{n}\TC(V;\Z/p)$ is concentrated in homotopy degrees $\leq \pdim(F)$ for all $n\in \Z$. Since the filtration is exhaustive, we also get that $\TC(V;\Z/p)$ is concentrated in degrees $\leq \pdim(F)$.
\end{proof}

The following consequence is Theorem~\ref{introthm:C} from the introduction.
\begin{cor}\label{cor:vanishing-on-valuation-rings}
Let $V$ be a valuation ring of characteristic $p$ with field of fractions $F$. Then $K_{i}(V; \Z/p) = 0$ for $i > \pdim(F)$.
\end{cor}
\begin{proof}
As $V$ is a local ring, the cyclotomic trace $K_{i}(V;\Z/p)\to TC_{i}(V;\Z/p)$ is injective by \cite[Thm.~D]{ClausenMathewMorrow} and the statement follows from the previous proposition.
\end{proof}

\begin{rmk}
In case $V=F$ is a field Corollary~\ref{cor:vanishing-on-valuation-rings} was proved
 for $\pdim (K)=0 $ in~\cite[Cor.~5.5]{Kratzer} and in \cite[Thm.~5.4]{Hiller} and in~\cite{GeisserLevine} in general.
\end{rmk}

\section{The main result}

\begin{lemma}\label{lem:non-vanishing}
Let $k$ be a field of characteristic $p$, and let $B$ be the $k$-algebra $B=k[x_{1}, \dots, x_{r}]/(x_{1}, \dots, x_{r})^{2}$. 
Let $y_{1}, \dots, y_{s} \in k$ be $p$-independent. Then the  symbol $\{y_{1}, \dots, y_{s}, 1+x_{1}, \dots, 1+x_{r}\} \in K_{s+r}(B)$ does not vanish in $K_{s+r}(B)/pK_{s+r}(B)$.
\end{lemma}
\begin{proof}
It suffices to check that  the image of the symbol under the Dennis trace map $K_{s+r}(B) \to \HH_{s+r}(B/\F_{p})$ does not vanish.
Note that $B \cong B_{0} \otimes_{\F_{p}} k$ where $B_{0} = \F_{p}[x_{1}, \dots, x_{r}]/(x_{1}, \dots, x_{r})^{2}$.
By the K\"unneth formula for Hochschild homology \cite[Prop.~9.4.1]{Weibel-homo} and the Hochschild--Kostant--Rosenberg theorem \cite[Thm.~9.4.7]{Weibel-homo} we have a natural isomorphism of graded rings
\[
\HH_{*}(B/\F_{p}) \cong \Omega^{*}_{k}  \otimes_{\F_{p}} \HH_{*}(B_{0}/\F_{p}).
\]
By \cite[Thm.~2.1]{GH-vorst} the Dennis trace maps the symbol $\{1+x_{1}, \dots, 1+x_{r}\} \in K_{r}(B_{0})$ to a non-zero element of $\HH_{r}(B_{0}/\F_{p})$. 
On the other hand, the Dennis trace is a map of graded rings, and the symbol $\{y_{1}, \dots, y_{s}\} \in K_{s}(k)$ is mapped to 
\[
d\log(y_{1}) \dots d\log(y_{s}) = (y_{1}\cdots y_{s})^{-1} dy_{1} \dots dy_{s} \in \Omega^{s}_{k},
\]
see the proof of \textit{loc.~cit.} and the references given there. The latter element does not vanish since $y_{1}, \dots, y_{s}$ are $p$-independent,
finishing the proof. 
\end{proof}

Recall that the \emph{homotopy $K$-theory} of a scheme $X$ is defined as
\[
KH(X) = \colim_{\Delta^\op} K(X\times\Delta^\bullet)  
\]
where $\Delta^\bullet$ is the cosimplicial scheme with $\Delta^p=\operatorname{Spec}(\mathbb{Z}[T_0,\ldots,T_p]/(\sum T_j-1))$ and $K(-)$ denotes the non-connective $K$-theory spectrum.

\begin{prop}\label{prop:vanishing}
Let $X$ be a noetherian $\F_p$-scheme. Then $KH_i(X;\Z/p) = 0$ for $i>\pdim(X)$.
\end{prop}
\begin{proof}
 We can assume that $d=\dim(X)\le \pdim(X)$ is finite and we use induction on $d$.
A zero-dimensional noetherian scheme is a finite disjoint union of schemes
$\Spec(A)$ where $A$ is artinian local with residue field $k$.
As homotopy $K$-theory is nilinvariant $KH_i(A;\mathbb Z /p)\cong KH_i(k;\mathbb Z /p)$ and
this group vanishes for $i>\pdim(k)$ by Corollary~\ref{cor:vanishing-on-valuation-rings}.

We proceed with the inductive step for $d>0$.
Again by Zariski descent we may assume $X=\Spec(A)$ where $A$ is a noetherian $\F_p$-algebra of finite Krull dimension $d$.
By cdh-descent \cite[Thm.~3.9]{Cisi13} and nil-invariance of $KH$ we may assume that $X$ is integral.
Let $\pi\colon X'\to X$ be a modification, i.e.~a proper birational morphism with $X'$ integral. 
There exists a closed subscheme $Y$ of $X$ with $\dim(Y)<\dim(X)$ such that $\pi$ is an isomorphism outside $Y$.
We obtain an abstract blow-up square
\[
\begin{tikzcd}
X' \arrow[d, "\pi"'] & \arrow[l]\arrow[d] Y'\\
X & \arrow[l] Y 
\end{tikzcd}
\]
inducing a long exact  sequence
\begin{multline*}
\ldots\to KH_{i+1}(Y';\Z/p) \to KH_i(X;\Z/p)\to\\\to  KH_i(Y;\Z/p)\oplus KH_i(X';\Z/p)\to KH_{i}(Y';\Z/p)\to\dots
\end{multline*}
Let $i>\pdim(X)$.
Since $\pdim(X)\geq \pdim(Y)$, the group $KH_i(Y;\Z/p)$ vanishes by the inductive hypothesis.
Note that Lemma~\ref{lem:p-dim-and-transcendence-degree} implies that $\pdim(X)\geq\pdim(X')$.
As $\pdim(X')\geq\pdim(Y')$, the inductive hypothesis yields $KH_{i+1}(Y';\Z/p)=KH_{i}(Y';\Z/p)=0$.
Hence, any modification $X'\to X$ induces an isomorphism $KH_i(X;\Z/p) \cong KH_i(X';\Z/p)$ for all integers $i>\pdim(X)$ and consequently an isomorphism
\begin{equation}\label{eqn:colimit-isomorphism}
 KH_i(X;\Z/p)  ~\cong \colim_{\substack{X'\to X \\ \text{modification}}} KH_i(X';\Z/p).
\end{equation}
There is a convergent Zariski descent spectral sequence
\[
E_2^{st}= H^s (X', a_\Zar KH_{-q}(-;\Z/p)) \Longrightarrow KH_{-s-t}(X';\Z/p)
\]
where $a_\Zar$ denotes Zariski sheafification. Note that  $\dim(X') \leq \dim(X)$ for every modification $X' \to X$ and
$E^{st}_2$ vanishes unless $0\le
s\le \dim(X')$.
Taking the filtered colimit as above of these uniformly bounded spectral sequences yields a convergent spectral sequence
\[
\colim_{\substack{X'\to X \\ \text{modification}}}  H^s ( X' , a_\Zar KH_{-t}(-;\Z/p) )
\Longrightarrow   \colim_{\substack{X'\to X \\ \text{modification}}}   KH_{-s-t}(X';\Z/p).
\]
We can use~\eqref{eqn:colimit-isomorphism} to identify the right-hand side with
$KH_{-s-t}(X;\Z/p)$ for $-s-t> \pdim(X) $.
We want to show that the left-hand side vanishes for all $s\in \Z$ and $-t > \pdim(X)$.
Consider the Zariski--Riemann space
\[
ZR(X) = \lim_{\substack{X'\to X \\ \text{modification}}} X'
\]
where the limit is formed in the category of locally ringed spaces (see~\cite[Def.~E.2.3]{FK}).
By \cite[Prop.~3.1.19]{FK} we can rewrite the left-hand side of the above spectral sequence as $H^s ( ZR(X) , \cF_{-t})$ where $\cF_t$ is the colimit of the sheaves $a_\Zar KH_{t}(-;\Z/p)$ on $X'_\Zar$.
Consider an integer $t > \pdim(X)$.
It suffices to show that the sheaf $\cF_t$ vanishes on the topological space $ZR(X)$.
This can be tested on stalks.
Hence, by \cite[Cor.~E.2.13]{FK} we must show that $\cF(V)$ vanishes for every valuation ring $V$ of $F$ where $F$ denotes the function field of $X$.
As (homotopy) $K$-theory commutes with filtered colimits, we have $\cF(V)\cong KH_t(V;\mathbb Z / p)$.
Now $\pdim(F)\leq \pdim(X)$ implies $t>\pdim(F)$. So the vanishing of $KH_t(V;\mathbb Z /
p)$ follows from Corollary~\ref{cor:vanishing-on-valuation-rings} and the following
Lemma~\ref{lem:kkhvalring}.
\end{proof}

\begin{lem}\label{lem:kkhvalring}
For a valuation ring $V$ and $m\ge 0$ the canonical map $K(V ) \to K(V[ X_1,\ldots , X_m
])$ is an equivalence. In
particular we get an equivalence $K(V)\simeq KH(V)$.
\end{lem}
\begin{proof}
  
We have to show that any element $\xi\in NK_i(V[X_1,\ldots , X_m]) $ vanishes. As $V$ is a
filtered colimit of noetherian integral subrings and as $K$-theory commutes with filtered
colimits of rings we can assume that there exists a noetherian ring $A\subset V$ and an
element $\xi_A\in NK_i (A [X_1,\ldots , X_m])$ mapping to $\xi$.
By~\cite[Prop.~6.4]{KerzStrunkTamme} there exists a projective birational morphism $X'\to
\Spec (A)$ such that $\xi_A$ maps to $0$ in $  NK_i (X'\times \mathbb A^m)$. As the
morphism $\Spec(V)\to \Spec(A)$ factors through $X'$ by the valuative criterion for
properness, we see that $\xi=0$.
\end{proof}

\begin{rmk}
  Lemma~\ref{lem:kkhvalring} was suggested by Christian Dahlhausen.
Kelly and Morrow also prove Lemma~\ref{lem:kkhvalring}  by a different method, see \cite[Thm.~3.3]{KellyMorrow}.
\end{rmk}

Let $n$ be an integer.
Recall that a ring $A$ is called $K_n$-regular if the canonical map $K_n(A)\to K_n(A[X_1,\ldots,X_m])$ is an isomorphism for all positive integers $m$, or equivalently, if $N^pK_n(A)=0$ (see \cite[Def.~III.3.4]{Kbook}) for all positive integers $p$. 
Vorst and van der Kallen proved that $K_n$-regularity implies $K_{n-1}$-regularity \cite[Cor.~2.1]{Vorst}.
In fact, they just consider the case $n\geq 1$ and the statement for all integers $n$ can be found in \cite[Thm.~V.8.6]{Kbook}.
Together with the spectral sequence 
\[
E^1_{st} = N^sK_t(A)  \Longrightarrow KH_{s+t}(A)
\]
from \cite[Thm.~IV.12.3]{Kbook} this implies that, if $A$ is $K_n$-regular, the canonical map
\[
 K_i(A) \to KH_i(A)
\]
is an isomorphism for all integers $i\leq n$ and surjective for $i=n+1$. The next result is Theorem~\ref{introthm:A} from the introduction.

\begin{thm}\label{thm:main}
Let $A$ be an excellent $\F_p$-algebra such that $[k(x):k(x)^p]<\infty$ for all points $x\in\Spec(A)$.
If $A$ is $K_{\pdim(A)+1}$-regular, then $A$ is regular.
\end{thm}

\begin{rmk}
Note that a reduced $\F_p$-algebra  $A$ which satisfies $[k(x):k(x)^p]<\infty$ for all
maximal points  $x\in\Spec(A)$  is excellent if and only if it is Frobenius finite, as
shown by Kunz and Datta--Smith~\cite[Cor.~2.6]{DaSm}.
\end{rmk}

\begin{proof}
First observe that we can assume without loss of generality that $A$ has finite Krull dimension.
Indeed, we have to show that the finite dimensional ring $A_\p$ is a regular local ring for all prime ideals
$\p\subset A$. But by \cite[Cor.~1.9]{Vorst} the localization $A_\p$ is
$K_{\pdim(A_\p)+1}$-regular.

We show the statement by induction on the finite Krull dimension $d$ of $A$.
For $d=0$ the noetherian ring $A$ is regular if it is reduced and this is immediately implied by $K_1$-regularity.
We proceed with the inductive step for $d>0$. By the above observation
we can assume that $A$ is local.

Next, we want to reduce to the case of complete local $\F_p$-algebras.
Let $A\to \widehat A$ be the completion at the maximal ideal $\fm$.
The ring $A$ is regular if and only if $\widehat A$ is regular.
In order to finish the reduction, we must show that $\widehat A$ is $K_{\pdim(\widehat A)+1}$-regular.
As $\pdim(\widehat A)\leq \pdim(A)$ by Lemma~\ref{lem:p-dim-and-dimension}, it suffices to show that $\widehat A$ is $K_{\pdim(A)+1}$-regular.
For an integer $q\geq 1$ and $i=\pdim(A)+1$ consider the commutative diagram
\begin{equation}\label{thm.comdiag}
\begin{tikzcd}[column sep=small]
N^qK_{i+1}(X) \arrow[d] \arrow[r]    & N^qK_i(\text{$A$ on
  $\fm$}) \arrow[r]\arrow[d,"\cong"'] & N^qK_i(A) \arrow[r]\arrow[d]&
N^qK_i(X) \arrow[d]  \\
 N^qK_{i+1}(\widehat X)  \arrow[r] & N^qK_i(\text{$\widehat A$ on $\fm$}) \arrow[r]          & N^qK_i(\widehat A) \arrow[r]& N^qK_i(\widehat X)  
\end{tikzcd}
\end{equation}
with exact rows, where $X= \Spec(A)\setminus \{\fm\}$ and $\widehat X = \Spec(\widehat  A)\setminus \{\fm\}$. 
We will show that the groups in the corners vanish.
Let $\p\not=\fm$ be a prime ideal of $A$. 
As before,  the local ring $A_{\p}$ is $K_{\pdim(A_{\p})+1}$-regular. As $\dim(A_{\p}) < \dim(A)$, the ring $A_{\p}$ is regular by the inductive hypothesis. So $X$ is a regular scheme.
As $A$ is excellent, the morphism $\Spec(\widehat A)\to \Spec(A)$ is regular. By \cite[IV, Scholie (7.8.3)(v)]{EGA} it follows that also $\widehat X$ is regular.
Hence, the groups in the corners of the above diagram involving $X$ and $\widehat X$ vanish.
By Thomason--Trobaugh excision \cite[Prop.~3.19]{Thomason-Trobaugh} for $N^qK$, the second vertical map is an isomorphism.
This implies that the ring $\widehat A$ is $K_{\pdim(A)+1}$-regular which finishes the reduction.
We can now assume that $A$ is a complete $\F_p$-algebra.

Let $k$ denote the residue field of $A$ and set $e=\pdim(A)$.
We have
\begin{equation}\label{eqn:dimension-in-the-main-theorem}
e = \pdim(k) + d
\end{equation}
by Lemma~\ref{lem:p-dim-and-dimension} which is a finite number as $\pdim(k)$ is finite by assumption.
Since $A$ is $K_{e+1}$-regular, the fibre of $K\to KH$ is $(e+1)$-connected by the
discussion preceding Theorem~\ref{thm:main}.
Hence the fibre of $K/p\to KH/p$ is $(e+1)$-connected as well and in particular
\[
 K_{e+1}(A;\Z/p) \to KH_{e+1}(A;\Z/p)
\]
is injective.
The group on the right-hand side vanishes by Proposition~\ref{prop:vanishing}, and consequently $K_{e+1}(A;\Z/p)=0$.

Consider a minimal set of generators $x_1,\ldots,x_r$ for the maximal ideal $\fm$ of $A$.
We have $r=\dim_k\fm/\fm^2\geq d=\dim (A)$ and $A$ is regular if and only if equality holds.
By Cohen's theorem~\cite[Thm.~28.3]{MatsumuraCA} the equicharacteristic complete local ring $A$ has a coefficient field, i.e.~the projection $A\to A/\fm=k$ admits a split $k\to A$.
Hence also $k\to A/\fm^2\to k$ is the identity and there is a surjection $k[X_1,\ldots,
X_r]\to A/\fm^2$ which induces an isomorphism $B=k[X_1,\ldots, X_r]/(X_{1}, \dots, X_{r})^{2} \cong A/\fm^2$.

Finally we will show that $K_{i}(A;\Z/p)\neq 0$ for all $i\in\{1,\ldots,\pdim(k)+r\}$.
This implies $\pdim(k)+r\leq e$ by the vanishing result from before and
 equation~\eqref{eqn:dimension-in-the-main-theorem} then gives $r\leq d$, whence the regularity of $A$.
In order to show the non-vanishing, let $y_{1}, \dots, y_{s} \in k$ be $p$-independent elements and consider the symbol $\bar\xi\in K_{s+r}(B)$ which is given by the image of $\{y_{1}, \dots, y_{s}, 1+X_{1}, \dots, 1+X_{r}\}$.
This has a preimage $\xi\in K_{s+r}(A)$ as $A\to B$ is surjective. 
Consider the diagram
\[
\begin{tikzcd}
\xi\in K_{s+r}(A) \arrow[r]\arrow[d] & K_{s+r}(A)/p\arrow[d]              \\
\bar\xi\in K_{s+r}(B) \arrow[r]          & K_{s+r}(B)/p         
\end{tikzcd}
\]
The image of $\xi$ in the lower right corner is non-trivial by Lemma~\ref{lem:non-vanishing}.
Hence the image of the element $\xi$ in $K_{s+r}(A)/p$ is non-trivial.
As the canonical map $K_{s+r}(A)/p\to K_{s+r}(A;\Z/p)$ is injective, we obtain $K_{s+r}(A;\Z/p)\neq 0$.
\end{proof}

In conjunction with Lemma~\ref{lem:p-dim-and-dimension} we obtain:

\begin{cor}
Let $k$ be a perfect field of positive characteristic and let $A$ be a $k$-algebra of finite type.
Suppose that $A$ is $K_{\dim(A)+1}$-regular. Then $A$ is a regular ring.
\end{cor}

We close this section by proving that in characteristic zero Vorst's conjecture can be
generalized from affine algebras over fields to all excellent rings. Recall that  Vorst's
conjecture in characteristic zero
is shown in~\cite{CHW-Vorst}. Our generalization is
based on the Hironaka--Artin algebraization of isolated singularities. The following result is Theorem~\ref{introthm:B} from the introduction.

\begin{thm}\label{thm:char0}
Let $A$ be an excellent noetherian ring of characteristic zero. If $A$ is
$K_{\dim(A)+1}$-regular, then $A$ is regular.
\end{thm}

\begin{proof}
  As in the proof of Theorem~\ref{thm:main} we can assume without loss of generality that
  $A$ has finite Krull dimension and that it is local. So we prove Theorem~\ref{thm:char0}
  by induction on $d=\dim(A)$. By the induction assumption the localization $A_g$ is regular for any $g\in
  A\setminus A^\times$, so $A$ has at most an isolated singularity at
  its maximal ideal $\fm$. Arguing as in diagram~\eqref{thm.comdiag} we can also assume
  that $A$ is complete. 
  
  Let $k\subset A$ be a field of coefficients for $A$~\cite[Thm.~28.3]{Matsumura}.
  Then by~\cite[$0_{\mathrm{IV}}$ Prop.~22.7.2]{EGA} the $k$-algebra $A$ satisfies the assumptions
  of  Hironaka--Artin algebraization \cite[Thm.~3.8]{Artin} so that 
 $A$ is the completion of a
  $k$-algebra $R$ of finite type at a maximal ideal $\fm$. We can assume that $R$ is regular away from the ideal $\fm$ and that $\dim(R)=\dim(A)$. Let $X$ be
  $\Spec R\setminus \{\fm \}$, $\widehat X =  \Spec (A) \setminus {\fm}$, and let $q\ge 1$. In the commutative diagram with exact rows
\begin{equation}\label{thm.comdiag2}
\begin{tikzcd}[column sep=small]
N^qK_{i+1}(X) \arrow[d] \arrow[r]    & N^qK_i(\text{$R$ on
  $\fm$}) \arrow[r]\arrow[d,"\cong"'] & N^qK_i(R) \arrow[r]\arrow[d]&
N^qK_i(X) \arrow[d]  \\
 N^qK_{i+1}(\widehat X)  \arrow[r] & N^qK_i(\text{$A$ on $\fm$}) \arrow[r]          & N^qK_i( A) \arrow[r]& N^qK_i(\widehat X)  
\end{tikzcd}
\end{equation}
the groups in the corners involving $X$ and $\widehat X$ vanish since these schemes are regular. The second vertical map
is an isomorphism by Thomason--Trobaugh excision \cite[Prop.~3.19]{Thomason-Trobaugh} for
$N^qK$. So for $i\le \dim(R)+1 = \dim(A)+1$ the vanishing of $N^qK_i(A)$ implies the
vanishing of  $N^qK_i(R)$. By~\cite[Thm.~0.1]{CHW-Vorst} the ring $R$ is regular and so is its completion $A$.
\end{proof}

\section{A one-dimensional analog of Vorst's conjecture in mixed  characteristic}

In this section we prove Theorem~\ref{introthm:mixed-char} (see
Theorem~\ref{thm:charmixed} below). The proof is based on
calculations involving de Rham--Witt complex in mixed
characteristic, which was introduced in \cite[Thm.~A]{HesselholtMadsen-drW}.
The key non-vanishing result is
Proposition~\ref{sec6:nonvanishing}.

\subsection{de Rham--Witt computations}

Let $p$ be an odd prime. 
For any commutative ring $R$ we denote by $R[\epsilon]$ the ring of dual numbers over $R$, i.e.~$R[\epsilon] \cong R[t]/(t^{2})$ with $\epsilon$ corresponding to the class of $t$.
Recall from \cite[Lemma~4.1.1]{HesselholtMadsen-drW} that every element of the ring $W_{n}(R[t])$ of $p$-typical Witt vectors of length $n$ may be written uniquely as a finite sum 
\begin{equation}
	\label{eq:HM-formula-Witt}
\sum_{j\in \N_{0}} a_{0,j}^{(n)} [t]_{n}^{j} + \sum_{s=1}^{n-1}\sum_{j \in I_{p}} V^{s}(a_{s,j}^{(n-s)}[t]_{n-s}^{j})
\end{equation}
with $a_{s,j}^{(n-s)} \in W_{n-s}(R)$. Here $I_{p}$ is the set of natural numbers not divisible by $p$
and for any ring $A$ the symbol $[\cdot]_{n} \colon A \to W_{n}(A)$ denotes the Teichm\"uller map.
The index $n$ will usually be clear from the context, and we often drop it from the notation to increase readability. 
With the same proof as in \cite{HesselholtMadsen-drW} one obtains the following lemma.
\begin{lemma}
	\label{lem:Witt-vectors-dual-numbers}
Any element in $W_{n}(R[\epsilon])$ may be written uniquely in the form
\[
a_{0,0}^{(n)} + a_{0,1}^{(n)}[\epsilon]_{n} + \sum_{s=1}^{n-1} V^{s}(a_{s,1}^{(n-s)}[\epsilon]_{n-s})
\]
with $a_{s,j}^{(n-s)} \in W_{n-s}(R)$.
The kernel of the canonical surjection $W_{n}(R[t]) \to W_{n}(R[\epsilon])$ consists of those elements \eqref{eq:HM-formula-Witt} for which $a^{(n-s)}_{s,j} = 0$ whenever $j =0$ or $j=1$.
\end{lemma}

We now assume that $R$ is a $\Z_{(p)}$-algebra. 

\begin{lemma}
	\label{lem:deRham-Witt-dual-numbers}
Any element in $W_{n}\Omega^{q}_{R[\epsilon]}$ can be written uniquely as a finite  sum of the form
\begin{equation}
	\label{eq:elements-in-drW-of-dual-numbers}
a_{0,0}^{(n)} + a_{0,1}^{(n)}[\epsilon]_{n} + b_{0}^{(n)} d[\epsilon]_{n} 
	+ \sum_{s=1}^{n-1} \left( V^{s}(a_{s,1}^{(n-s)}[\epsilon]_{n-s} ) + dV^{s}( b_{s}^{(n-s)}[\epsilon]_{n-s} ) \right)
\end{equation}
where $a_{s,i}^{(n-s)} \in W_{n-s}\Omega^{q}_{R}$ and $b_{s}^{(n-s)} \in W_{n-s}\Omega^{q-1}_{R}$.
In other words, as abelian group
\begin{equation}
	\label{eq:deRham-Witt-direct-sum-decomposition}
W_{n}\Omega^{q}_{R[\epsilon]} \cong 
	W_{n}\Omega^{q}_{R} \oplus W_{n}\Omega^{q}_{R} \oplus W_{n}\Omega^{q-1}_{R}
	\oplus \bigoplus_{s=1}^{n-1} \left( W_{n-s}\Omega^{q}_{R} \oplus W_{n-s}\Omega^{q-1}_{R} \right).
\end{equation}
The structure maps of the pro-system $W_{\bullet}\Omega^{q}_{R[\epsilon]}$ are induced by the ones of $W_{\bullet}\Omega^{q}_{R}$.
On the first summand $W_{n}\Omega^{q}_{R}$ the structure maps $d$, $F$, and $V$ are given by the underlying maps of $W_{\bullet}\Omega^{*}_{R}$, on the other summands they are given in Table~\ref{table:de-rham-witt}.

The product is given as follows.  On the first summand, the product is the one from $W_{n}\Omega^{*}_{R}$. The product of two summands from \eqref{eq:elements-in-drW-of-dual-numbers} each of which has an $[\epsilon]$ vanishes. Finally, for $a' \in W_{n}\Omega^{q'}_{R}$ the left multiplication with $a'$ is given in Table~\ref{table:product-dRW}.
\end{lemma}

\begin{table}[t]
	\renewcommand{\arraystretch}{1.3}
\begin{small}
\begin{tabular}{cccc}
\toprule
	$x$			& $d(x)$ 							& $F(x)$ 				& $V(x)$ \\ 
\midrule
$a[\epsilon]$		& $(da)[\epsilon] + (-1)^{q}d[\epsilon]$	& $0$				& $V(a[\epsilon])$ \\
$bd[\epsilon]$		& $(db)d[\epsilon]$ 					& $0$				& $(-1)^{q-1}pdV(b[\epsilon]) + (-1)^{q}V((db)[\epsilon])$ \\
$V^{s}(a[\epsilon])$	& $dV^{s}(a[\epsilon])$				& $pV^{s-1}(a[\epsilon])$ 	& $V^{s+1}(a[\epsilon])$ \\
$dV^{s}(b[\epsilon])$	& $0$							& $dV^{s-1}(b[\epsilon])$	& $pdV^{s+1}(b[\epsilon])$ \\
\bottomrule
\end{tabular}
\end{small}
\caption{Structure maps on $W_{n}\Omega^{q}_{R[\epsilon]}$}
\label{table:de-rham-witt}
\end{table}

\begin{table}[t]
	\renewcommand{\arraystretch}{1.3}
\begin{small}
\begin{tabular}{ccccc}
\toprule
$x$			& $a[\epsilon]$ 		& $bd[\epsilon]$ 	& $V^{s}(a[\epsilon])$  			& $dV^{s}(b[\epsilon])$ \\ 
\midrule
$a'\cdot x$	& $(a'a)[\epsilon]$ 	& $(a'b)d[\epsilon]$	& $ V^{s}( F^{s}(a')a [\epsilon])$ 	&  $(-1)^{q+qq'} V^{s}(bF^{s}(da')[\epsilon]) $   \\
			&				&				&							& $+ (-1)^{qq'}dV^{s}( bF^{s}(a')[\epsilon])$ \\
\bottomrule
\end{tabular}
\end{small}
\caption{Multiplication by $a' \in W_{n}\Omega^{q'}_{R}$ on $W_{n}\Omega^{q}_{R[\epsilon]}$}
\label{table:product-dRW}
\end{table}

\begin{proof}
By \cite[Lemma~2.4]{GeisserHesselholtComplete} the canonical map $W_{n}\Omega^{*}_{R[t]} \to W_{n}\Omega^{*}_{R[\epsilon]}$ is surjective with kernel the dg-ideal generated by the kernel of the map $W_{n}(R[t]) \to W_{n}(R[\epsilon])$. 
Using the description of this ideal in Lemma~\ref{lem:Witt-vectors-dual-numbers},
the lemma now follows directly from the corresponding description of the de Rham--Witt complex of the polynomial ring $R[t]$ in \cite[Thm.~4.2.8]{HesselholtMadsen-drW}. 
\end{proof}

For ease of notation, we introduce the following abbreviation. For any ring $k/\mathbb F_p$ we set 
\[
B(k) = W_{2}(k)[\epsilon]/(p\epsilon).
\]
For example, $B(\F_{p}) = \Z[\epsilon]/(p, \epsilon)^{2}$.

\begin{prop}
	\label{prop:non-vanishing-in-dRW}
For $n \geq 2$, the element $d[1+p]_{n}d[1+\epsilon]_{n}$ does not vanish in $W_{n}\Omega^{2}_{B(\F_{p})}/p$.
\end{prop}

\begin{proof}
We write $\omega_{n,R}^{*}$ for the quotient of $W_{n}\Omega^{*}_{R[\epsilon]}$ 
by the dg-ideal given by  $ \bigoplus_{s=1}^{n-1} \left( W_{n-s}\Omega^{*}_{R} \oplus W_{n-s}\Omega^{*-1}_{R} \right)$  in \eqref{eq:deRham-Witt-direct-sum-decomposition}.

\begin{claim}
	\label{claim:d-epsilon}
We have $d[1+\epsilon]_{n} = d[\epsilon]_{n}$ in $\omega^{1}_{n,R}$.
\end{claim}
\begin{proof}
In $W_{n}(R[\epsilon])$ we have $[1+\epsilon]_{n} = 1 + [\epsilon]_{n} + V(x)$ for some element $x \in W_{n-1}(R[\epsilon])$. From Lemma~\ref{lem:Witt-vectors-dual-numbers} we see that $V(x)$ is of the form $V(x_{0}) + \sum_{s=1}^{n-1} V^{s}(x_{s,1}[\epsilon]_{n-s})$ with $x_{0} \in W_{n-1}(R)$ and $x_{s,1} \in W_{n-s}(R)$. Since the ring homomorphism $W_{n}(R[\epsilon]) \to W_{n}(R)$ that is induced by $\epsilon \mapsto 0$ must  send $V(x)$ to $0$, we find that in fact $V(x_{0}) = 0$. Thus 
\[
[1+\epsilon]_{n} = 1 + [\epsilon]_{n} + \sum_{s=1}^{n-1} V^{s}(x_{s,1}[\epsilon]_{n-s}).
\]
Applying $d$, the claim follows.
\end{proof}

\begin{claim}
	\label{claim:vanishing-of-[p]}
Multiplication by $[p]_{n}$ is the zero map on $W_{n}\Omega^{1}_{\Z_{(p)}}$ and hence also on $W_{n}\Omega^{1}_{R}$ where $R$ is any quotient of $\Z_{(p)}$.
\end{claim}
\begin{proof}
By \cite[Ex.~1.2.4]{HesselholtMadsen-drW} we have  isomorphisms
\begin{equation}
	\label{eq:drW-isos-Zp}
W_{n}(\Z_{(p)}) \cong \bigoplus_{i=0}^{n-1} \Z_{(p)} \cdot V^{i}(1) \quad \text{ and }  \quad
W_{n}\Omega^{1}_{\Z_{(p)}} \cong \bigoplus_{i=1}^{n-1} \Z/p^{i}\Z \cdot dV^{i}(1)
\end{equation}
and for $i,j \in \{ 0, \dots, n-1\}$ we have  $V^{i}(1)dV^{j}(1) = p^{i}dV^{j}(1)$ if $i < j $  and $=0$ else.
For every element $x \in \Z_{(p)}$ we have
\begin{equation}
	\label{eq:coordinates}
[x]_{n} = x\cdot [1]_{n} + \sum_{i=1}^{n-1} p^{-i}(x^{p^{i}} - x^{p^{i-1}}) \cdot V^{i}( [1]_{n-i})
\end{equation}
as one checks by computing the ghost components of both sides.
By these formulas,  the action of $[p]_{n}$ on $dV^{j}(1)$ is given by
\[
[p]_{n}\cdot dV^{j}(1) = p dV^{j}(1) + \sum_{i=1}^{j-1} (p^{p^{i}} - p^{p^{i-1}}) dV^{j}(1) = p^{p^{j-1}} dV^{j}(1) = 0
\]
as $p^{j-1} \geq j$ for every $j \geq 1$. This proves the claim.
\end{proof}

\begin{claim}
	\label{claim:retract}
For $R = \Z_{(p)}$ or some quotient of it, the projection $W_{n}\Omega^{2}_{R[\epsilon]} \to \omega_{n,R}^{2}$ factors through the 
 canonical surjection $W_{n}\Omega^{2}_{R[\epsilon]} \to W_{n}\Omega^{2}_{R[\epsilon]/(p\epsilon)}$. 
\end{claim}
Note that for $R = \Z/(p^{2})$, $R[\epsilon]/(p\epsilon) = B(\F_{p})$.
\begin{proof}
The kernel of the surjection $W_{n}\Omega^{*}_{R[\epsilon]} \to W_{n}\Omega^{*}_{R[\epsilon]/(p\epsilon)}$ is the dg-ideal generated by the kernel of the map $W_{n}(R[\epsilon]) \to W_{n}(R[\epsilon]/(p\epsilon))$. An element of this kernel is of the form $\sum_{i=0}^{n-1} V^{i}( [x_{i}p\epsilon]_{n-i} )$ with $x_{i} \in R$.
The dg-ideal generated by the elements $V^{i}( [x_{i}p\epsilon]_{n-i})$ for $i > 0$ lies in the kernel of $W_{n}\Omega^{*}_{R[\epsilon]} \to \omega^{*}_{n,R}$ by definition. 

It remains to show that the  dg-ideal generated by $[p\epsilon]_{n}$  in  $\omega^{*}_{n,R}$ vanishes in degree 2.
We have $d[p\epsilon] = [p]d[\epsilon] + [\epsilon]d[p]$.
Hence the vanishing of $W_{n}\Omega^{2}_{R}$ \cite[Ex.~1.2.4]{HesselholtMadsen-drW}, Claim~\ref{claim:vanishing-of-[p]}, and Lemma~\ref{lem:deRham-Witt-dual-numbers} together imply that the product of $d[p\epsilon]$ with any 1-form in $\omega^{1}_{n,R}$ vanishes. 
Since every 2-form in $\omega^{*}_{n,R}$ is a multiple of $d[\epsilon]$, also the product of $[p\epsilon]$ with any element of $\omega^{2}_{n,R}$ vanishes. This finishes the proof of the claim.
\end{proof}

\begin{claim}
	\label{claim:mod-p-iso}
The natural map $W_{n}\Omega^{*}_{\Z_{(p)}} / p \to W_{n}\Omega^{*}_{\Z/(p^{2})}/p$ is an isomorphism.
\end{claim}
\begin{proof}
The kernel of $W_{n}\Omega^{*}_{\Z_{(p)}} \to W_{n}\Omega^{*}_{\Z/(p^{2})}$ is the dg-ideal generated  by the kernel $W_{n}(p^{2}\Z_{(p)})$ of the canonical map  $W_{n}(\Z_{(p)}) \to W_{n}(\Z/(p^{2}))$. It is enough to show that  $W_{n}(p^{2}\Z_{(p)}) \subset pW_{n}(\Z_{(p)})$.
The ideal $W_{n}(p^{2}\Z_{(p)})$ is additively generated by elements of the form $V^{i}([xp^{2}]_{n-i})$.
Set $m = n-i$. In the expression \eqref{eq:coordinates} of $[xp^{2}]_{m}$ the coefficient of $[1]_{m}$ is $xp^{2}$. For $i > 0$, the coefficient of $V^{i}([1]_{m-i})$ is divisible by $p^{-i} (p^{2})^{p^{i-1}} = p^{2p^{i-1} - i}$. Since $2p^{i-1} - i \geq 1$ for all $i \geq 1$, it is divisible by $p$. Thus $[xp^{2}]_{m}$ is a multiple of $p$ and so is any of its Verschiebungen.
\end{proof}

\begin{claim}	
	\label{claim:non-vanishing}
The element $d[1+p]_{2} \in W_{2}\Omega^{1}_{\Z_{(p)}}$ is not divisible by $p$.
\end{claim}
\begin{proof}
Using the expression  \eqref{eq:coordinates} we have
\[
d[1+p]_{2} = p^{-1}( (1+p)^{p} - (1+p) ) dV(1).
\]
But $p^{-1}( (1+p)^{p} - (1+p) ) 
= \sum_{k=1}^{p}  \binom{p}{k}p^{k-1} - 1 \equiv -1 \pmod p$. In view of the second isomorphism in~\eqref{eq:drW-isos-Zp} this finishes the proof.
\end{proof}

We now finish the proof of Proposition~\ref{prop:non-vanishing-in-dRW}.
In order to show that  $d[1+p]_{n}d[1+\epsilon]_{n}$ does not vanish in $W_{n}\Omega^{2}_{B(\F_{p})}/p$, it suffices by Claim~\ref{claim:retract} to show that the image of $d[1+p]_{n}d[1+\epsilon]_{n}$ in $\omega^{2}_{n,\Z/(p^{2})}/p$ does not vanish.
By Claim~\ref{claim:d-epsilon} this image coincides with the image of $d[1+p]_{n}d[\epsilon]_{n}$. By Lemma~\ref{lem:deRham-Witt-dual-numbers} it suffices to show that $d[1+p]_{n}$ does not vanish in $W_{n}\Omega^{1}_{\Z/(p^{2})} / p$; equivalently, $d[1+p]_{n} \not= 0$ in $W_{n}\Omega^{1}_{\Z_{(p)}}/p$ by Claim~\ref{claim:mod-p-iso}. It is clearly enough to show this non-vanishing for $n=2$, which is done in Claim~\ref{claim:non-vanishing}.
\end{proof}

\begin{prop}\label{sec6:nonvanishing}
Let $k$ be a $\F_{p}$-algebra. The canonical map $W_{n}\Omega^{q}_{B(\F_{p})}/p \to W_{n}\Omega^{q}_{B(k)}/p$ is injective for every $n$ and $q$. In particular,
 $d[1+p]_{n}d[1+\epsilon]_{n}$ does not vanish in $W_{n}\Omega^{2}_{B(k)}/p$ for $n\geq 2$.
\end{prop}

\begin{proof}
The second part of Proposition~\ref{sec6:nonvanishing} follows from the first part and
Proposition~\ref{prop:non-vanishing-in-dRW}.

Write $F$ for the functor $k \mapsto W_{n}\Omega^{q}_{B(k)}/p$ from rings over $\mathbb F_p$ to abelian groups.
This functor commutes with filtered colimits.
Recall from \cite[Thm.~2.4]{vdKallen} or \cite[Prop.~6.9, Thm.~9.2]{Borger} that if $R \to S$ is an \'etale covering, i.e.~\'etale and faithfully flat, then so is $W_{n}(R) \to W_{n}(S)$. 
Assume that  $k \to \ell$ is an \'etale covering. Then  $B(k) \to B(\ell)$ is also an \'etale covering by base change.
It follows from~\cite[Thm.~C]{Hesselholt-big} that the canonical map 
\[
W_{n}(B(\ell)) \otimes_{W_{n}(B(k))} W_{n}\Omega^{q}_{B(k)}/p   \to   W_{n}\Omega^{q}_{B(\ell)}/p
\]
is an isomorphism. Using the result of van der Kallen and Borger for the map $B(k) \to B(\ell)$, we get that  $W_{n}(B(k)) \to W_{n}(B(\ell))$ is an \'etale covering.
Together these results imply  that $F(k) \to F(\ell)$ is injective.
In particular, $F(k) \to F(\ell)$ is injective for a finite separable field extension $k \to \ell$. Since $F$ preserves filtered colimits,  the map $F(\F_{p}) \to F(\bar \F_{p})$ is injective, where $\bar\F_{p}$ is an algebraic closure of $\F_{p}$.

To prove injectivity for a general  $\F_{p}$-algebra $k$, it is now enough to show that $F(\bar \F_{p}) \to F( k \otimes_{\F_{p}} \bar\F_{p})$ is injective. 
We can thus assume that $k$ is an $\bar\F_{p}$-algebra. Write $k$ as a filtered colimit of finitely generated $\bar\F_{p}$-algebras $A_{i}$. As $\bar \F_{p}$ is algebraically closed, each map $\bar\F_{p} \to A_{i}$ has a section, and hence $F(\bar\F_{p}) \to F(A_{i})$ is injective. As a filtered colimit of injective maps, $F(\bar\F_{p}) \to F(k)$ is then also injective.
\end{proof}

\subsection{The mixed characteristic result}

The next proposition is a special case of~\cite[Prop.~B.1.1.]{GeisserHesselholtCycles}.

\begin{prop}\label{prop:dlog}
Let $p\neq 2$ be a prime and $A$ a $\Z_{(p)}$-algebra.
The natural map
\[
 \begin{array}{rcl}
 \dlog[-]_n\colon A^\times &\to     & W_n\Omega^1_A  \vspace{1ex}   \\
                 x        &\mapsto & [x]_n^{-1} d[x]_n
 \end{array}
\]
satisfies the Steinberg relation $\dlog[x]_n \dlog[1-x]_n=0$ and hence induces a natural map
\[
 \dlog[-]_n\colon K_2^M(A) \to W_n\Omega^2_A
\]
where the Milnor $K$-group $K_2^M(A)$ is defined as $A^\times\otimes A^\times$ modulo the subgroup generated by $a\otimes(1-a)$ for units $a$ and $1-a$. 
\end{prop}

The following result is Theorem~\ref{introthm:mixed-char} from the introduction.

\begin{thm}\label{thm:charmixed}
Let $A$ be an excellent noetherian ring with $\dim(A)\leq 1$ such that $A/\fm$ is perfect of characteristic $p> 2$ for every maximal ideal $\fm \subset A$.
If $A$ is $K_{2}$-regular, then $A$ is regular.
\end{thm}

\begin{proof}
Let $A$ be a ring satisfying the assumptions of the theorem.
We must show that $A$ is regular.
As in the proof of Theorem~\ref{thm:main} we can assume without loss of generality that $A$ is local with maximal ideal $\fm$.

Once we prove that the strict henselization $A^\sh$ is regular, we can deduce the regularity of $A$, see \cite[\href{https://stacks.math.columbia.edu/tag/06LN}{Lemma 06LN}]{stacks-project}.
As the canonical map $A\to A^\sh$ is a filtered colimit of \'etale morphisms by \cite[\href{https://stacks.math.columbia.edu/tag/04GN}{Lemma 04GN}]{stacks-project}, van der Kallen's result on the $NK_n$-groups \cite[Thm.~3.2]{vdKallen} implies that $A^\sh$ is still $K_2$-regular.
Moreover, $A^\sh$ is still excellent by \cite[Cor.~5.6]{Greco}.
Hence we can assume without loss of generality that $A$ is strict henselian and excellent.

Arguing exactly as in the proof of Theorem~\ref{thm:main}, we reduce to the case of $A$ being a complete local ring with $\dim(A)\leq 1$ and maximal ideal $\fm$, algebraically closed residue field $k=A/\fm$ of characteristic $p> 2$ and quotient field $F$ of characteristic zero.

\begin{claim}\label{sec6.claimzero}
$K_2(A)/p =0. $
\end{claim}
\begin{proof}[Proof of Claim~\ref{sec6.claimzero}]
By assumption, $A$ is a complete noetherian local ring.
The $K_1$-regularity implies that $A$ is also reduced.
Thus the normalization $A\to\tilde A$ is finite (see e.g.~\cite[\href{https://stacks.math.columbia.edu/tag/032Y}{Lemma 032Y}]{stacks-project}) and we obtain an abstract blow-up square
\[
\begin{tikzcd}
\Spec(\tilde A) \arrow[d] & Y'    \arrow[d]\arrow[l]          \\
\Spec(A)           & \Spec(k)\arrow[l]
\end{tikzcd}
\]
with $Y'_\red=\Spec(k)\times\ldots\times \Spec(k)$ as $k$ is algebraically closed.
Note that $\tilde A$ is regular.
Descent for $KH(-;\Z/p)$, see \cite[Thm.~3.9]{Cisi13}, yields an  exact sequence
\[
\ldots\to KH_3(Y';\Z/p) \to KH_2(A;\Z/p) \to  KH_2(\tilde A;\Z/p) \oplus KH_2(k;\Z/p) \to\ldots
\]
The group $KH_3(Y';\Z/p)= KH_3(Y'_\red;\Z/p)= K_3(Y';\Z/p)$ and the group $KH_2(k;\Z/p)= K_2(k;\Z/p)$ vanish by Corollary~\ref{cor:vanishing-on-valuation-rings}.

We have an injection $K_2(A)/p = KH_2(A)/p \hookrightarrow KH_2(A;\Z/p)$ where the first isomorphism uses the $K_2$-regularity of $A$, see \cite[Cor.~IV.12.3.2]{Kbook}.
Hence, the above exact sequence implies that the composite map
\[
 K_2(A)/p \hookrightarrow KH_2(A;\Z/p) \hookrightarrow KH_2(\tilde A;\Z/p) = K_2(\tilde A)/p
\]
is injective and it suffices to show that $K_2(\tilde A)/p=0$.

Consider the diagram
\[
\begin{tikzcd}
&K_2(\tilde A)/p \arrow[d, hook] \arrow[r]&  K_2(F)/p   \arrow[d, hook]         \\
K_3(k;\Z/p)\arrow[r] &K_2(\tilde A;\Z/p)\arrow[r]           & K_2( F;\Z/p)
\end{tikzcd}
\]
where the bottom horizontal line is part of the exact localization sequence of the
discrete valuation ring $\tilde A$.
As $K_3(k;\Z/p)=0$ by Corollary~\ref{cor:vanishing-on-valuation-rings}, we deduce
that the top horizontal map is injective. So to finish the proof of the claim it suffices to show that $K_2(F)/p= 0$.

By the Merkurjev--Suslin theorem \cite{MerkurjevSuslin}, the norm residue homomorphism
\[
K_2(F)/p \xrightarrow{\cong} H^2(F,\mu_p^{\otimes 2}) 
\]
is an isomorphism.
The group on the right-hand side vanishes as the field $F$ has property ($C_1$) by
\cite[II.3.3.(c)]{SerreGalois} invoking Lang's theorem and hence has cohomological dimension $\leq 1$ by \cite[II.3.2,~Cor.~to~Prop.~8]{SerreGalois}. 
Hence we get $K_2(A)/p=0$ as desired.
This shows the claim.
\end{proof}

Consider a minimal set of generators $x_1,\ldots,x_r$ for the maximal ideal $\fm$ of $A$.
The ring $A$ is regular if and only if $r=1$.
We consider two cases.

Suppose first that $p\in \fm^2$ and consider the $\F_p$-algebra $\bar A:=A/p$.
Then the images $\bar x_1,\ldots,\bar x_r$ are still a minimal set of generators for the maximal ideal $\bar\fm$ of $\bar A$.
Set $\bar B:= \bar A /\bar \fm^2$.
Arguing analogously as in the proof of Theorem~\ref{thm:main}, the image of the symbol $\{1+x_1,\ldots,1+x_r\}\in K_{i}(\bar B)$ in $K_{i}(\bar B)/p$ is non-zero and hence $K_{i}(A)/p\neq 0$ for all $i\in\{1,\ldots,r\}$.
This implies $r=1$ by Claim~\ref{sec6.claimzero} and $A$ is regular.

For the second case suppose that $p\notin \fm^2$. Then we can assume without loss of
generality  that $x_1=p$.
We argue by contradiction and suppose that $r\geq 2$. 
By Cohen's theorem~\cite[Thm.~28.3]{MatsumuraCA} the complete local ring $A$ has a coefficient ring $W(k)$.
Hence, the canonical map $X_i\mapsto x_i$ induces an isomorphism $W(k)[X_2,\ldots, X_r]/(p,X_2,\ldots, X_r)^2\cong A/\fm^2$.
This ring canonically surjects onto the ring $W_2(k)[X]/(p,X)^2\cong
W_{2}(k)[\epsilon]/(p\epsilon)$ which we denote by $B(k)$.
Analogously to the first case, it suffices to show that the image of the symbol $\{1+p, 1+\epsilon\}$ in $K_2(B(k))/p$ does not vanish.
By \cite{vanderKallen} and as the residue field $k$ is infinite, the canonical map
$K^M_2(B(k))\to K_2(B(k))$ is an isomorphism.
Choose some integer $n\geq 2$.
The $\dlog$-map from Proposition~\ref{prop:dlog} sends the symbol $\{1+p, 1+\epsilon\}$ to the element $[1+p]_{n}^{-1}d[1+p]_{n}[1+\epsilon]_{n}^{-1}d[1+\epsilon]_{n}$ in $W_n\Omega^2_{B(k)}/p$ which does not vanish by Proposition~\ref{sec6:nonvanishing}.
Hence $K_2(A)/p \neq 0$ which contradicts Claim~\ref{sec6.claimzero}.
This implies that $A$ is regular.
\end{proof}

\appendix

\section{The Cartier isomorphism for valuation rings \\  \normalfont (after  Ofer Gabber)}

In this appendix we present a detailed account of results of Gabber about valuations rings
in positive characteristic. In particular we construct the Cartier isomorphism for these
rings. The exposition is based on \cite{Gabber_letter} and~\cite{GabberRamero}.

\subsection{Elementary extensions}\label{subsec:elem}

Let $V\subset W$ be an extension of integral domains of characteristic $p$ such that $W^p\subset V$.
For such an extension we get a $V$-linear ``inverse Cartier'' operator
\begin{align}\label{eq.invcart}
\gamma_{W/V}: V\otimes_W \Omega_{W/V}^i &\xrightarrow{} Z\Omega^i_{W/V}/B\Omega^i_{W/V} \\ \notag
v\otimes b dy_1\wedge \ldots \wedge dy_i &\mapsto v b^p y_1^{p-1}dy_1\wedge \ldots  \wedge y_i^{p-1}dy_i,
\end{align}
see \cite[Sec.~7]{Katz-nilpotent}. Here $V$ becomes a $W$-module via the Frobenius map.
In case $V=W^p$ this map can be identified with the standard $W$-linear map on absolute forms,
discussed in Section~\ref{sec:deriform},
\[
\gamma_W:\Omega_{W}^i \to Z \Omega^i_{W}/B\Omega^i_{W},
\]
where the Frobenius induces the  $W$-module structure on the right.

We say that
the  extension $V\subset W$ is {\em elementary} if there exists a finite $p$-basis $x_1,\ldots ,
x_r\in W$ of $W/V$, see Definition~\ref{sec2.defpbase}. Note that $x_1,\ldots ,
x_r\in W$ form a $p$-basis of $W/V$ if $W=V[x_1, \ldots , x_r]$ and if these elements form
a $p$-basis of the extension of quotient fields $Q(W)/Q(V)$.
 An elementary extension is a flat local complete
intersection homomorphism of rings, since we have the presentation
\begin{equation}\label{app.eq1}
  W\cong V[X_1,\ldots , X_r]/(X_1^p-x_1^p,
  \ldots , X_r^p-x_r^p).
\end{equation}
   This presentation also implies that $\Omega_{W/V}$ is a free
$W$-module with basis $dx_1, \ldots , dx_r$.

By $\bbL_{W/V}$ we denote the cotangent complex of the ring extension $V\subset W$
and by $\bbL_V$ we denote the cotangent complex of $V$ over $\mathbb F_p$, see Section~\ref{sec:deriform}. 
\begin{prop}\phantomsection\label{app.propdiff} 
  \begin{itemize}
    \item[(i)]
  If the extension $V\subset W$ is elementary, then $\bbL_{W/V}$ is concentrated in 
  degrees $0$ and $1$, and $ H_i(\bbL_{W/V})$ is a flat $W$-module for $i\in \{0,1\}$.
\item[(ii)] If $V\subset V^{1/p}$ is a filtered colimit of elementary extensions $V\subset
  W$, then   $\bbL_{V}\simeq
  \Omega_{V}[0]$ and $\Omega_V$ is a flat $V$-module.
  \end{itemize}  
\end{prop}

\begin{proof}
Part~(i) is clear from the presentation~\eqref{app.eq1} and \cite[Cor.~III.3.2.7]{Illusie-def}. The second
statement of part~(ii) is clear from part~(i) as the extension $V\subset V^{1/p}$ is isomorphic to the extension
$V^p\subset V$ via the Frobenius map and as $\Omega_V=\Omega_{V/V^p} $.

So it remains to
show that $\bbL_{V}\simeq
\Omega_{V}[0]$ under the assumption of part~(ii). Note that then $V\subset V^{1/p}$ is faithfully flat. By part~(i) the cotangent complex
  $\bbL_{V^{1/p}/V}$ as well as the isomorphic cotangent complex  $\bbL_{V^{1/p^2}/V^{1/p}}$  are concentrated in degrees $0$ and $1$. By the exact triangle
\[
\bbL_{V^{1/p}/V }\otimes_{V^{1/p}} V^{1/p^2} \to  \bbL_{V^{1/p^2}/V } \to \bbL_{V^{1/p^2}/V^{1/p}}
\]
also  $\bbL_{ V^{1/p^2}/V}$ is  concentrated in degrees $0$ and $1$. Arguing
inductively this is also true for the faithfully flat extension $V\subset V^{1/p^\infty}$.

Note that $\bbL_{V^{1/p^\infty}}$ is concentrated in degree zero, since the ring
$V^{1/p^\infty}$ is perfect~\cite[6.5.13(i)]{GabberRamero}. So we conclude by the exact triangle
\[
\bbL_V \otimes_V V^{1/p^\infty} \to \bbL_{V^{1/p^\infty}} \to \bbL_{V^{1/p^\infty}/V}.\qedhere
\]
\end{proof}

\begin{prop}[Cartier isomorphism]\phantomsection\label{app.carprop} 
  \begin{itemize}
    \item[(i)]
  If the extension $V\subset W$ is elementary, then the morphism $\gamma_{W/V}$ of
  \eqref{eq.invcart} is an isomorphism.
\item[(ii)] If $V\subset V^{1/p}$ is a filtered colimit of elementary extensions $V\subset
  W$, then the map $\gamma_V$ is an isomorphism. 
  \end{itemize}
\end{prop}

 The isomorphism $\gamma$ is usually denoted by $C^{-1}$ and called the \emph{inverse Cartier operator}.

 \begin{proof}
   Part (ii) is an immediate consequence of part (i).
   For part (i) one
  reduces to $r=1$ as in the proof of~\cite[Thm.~7.2]{Katz-nilpotent}. Then one only has to consider $i=0$ and $i=1$. For $i=1$
  a $V$-basis of the left side of \eqref{eq.invcart} is given  by
  $dx_1$. A $V$-basis of $ B\Omega^1_{W/V}$ is given by $dx_1, x_1 dx_1 , \ldots ,
  x_1^{p-2}dx_1$, so a $V$-basis of the right side of \eqref{eq.invcart} is induced by
  $x_1^{p-1}dx_1=\gamma_{W/V} (dx_1)$.
\end{proof}

\subsection{Purely inseparable extensions of valued fields}

Let $K\subset K'$ be an extension of valued fields of characteristic $p$ with $(K')^p\subset K$. Let $V\subset
V'$ be the corresponding extension of valuation rings. A subextension of rings $V\subset
W\subset V'$ is called elementary if the extension $V\subset W$ is elementary in the sense
of  Subsection~\ref{subsec:elem}.

\begin{thm}[Gabber]\label{thm.elemext}
  The set $\mathcal S$ of
elementary extensions $V\subset W\subset V'$ is directed by inclusion
and $$\bigcup_{W\in \mathcal S}
W=V'.$$
\end{thm}

Combining Theorem~\ref{thm.elemext} with Propositions~\ref{app.propdiff} and~\ref{app.carprop} we obtain:

\begin{cor}\label{app.main.cor}
	\label{cor:Cartier-iso-valuation-ring}
        Let $V$ be a valuation ring of characteristic $p$. Then the following hold.
        \begin{itemize}
        \item[(i)] $\Omega_V$ is a flat $V$-module.
          \item[(ii)] $\bbL_V$ is concentrated in degree zero.
          \item[(iii)] The ``inverse Cartier'' operator $\gamma_V$ from Subsection~\ref{subsec:elem} is
            an isomorphism.
        \end{itemize}
      \end{cor}

      \begin{rmk}
Parts (i) and (ii) of Corollary~\ref{app.main.cor} are shown in Theorem~6.5.12 and Corollary~6.5.21 in~\cite{GabberRamero} using related techniques. These techniques are
extended in~\cite{Gabber_letter} to prove Theorem~\ref{thm.elemext} and
Corollary~\ref{app.main.cor}. 
        \end{rmk}

We repeatedly need the following well-known result about finite extensions of valuation
rings, see Sections VI.8.3 and VI.8.5 in~\cite{Bourbaki_comalg}. 
\begin{lem}\label{app:lembour}
Let $V\subset V'$ be an extension of valuation rings with the above properties and with
$q=[K':K]$ finite. Let $f$ be
the degree of the residue field extension and let $e=[|(K')^\times|:|K^\times |]$ be the
ramification index. Then
\begin{itemize}
\item[(i)] $q\ge e f$, $e|q$ and $f|q$,
  \item[(ii)] if $f=q$ the extension $ V'/V$ is finite and $ V' \mathfrak m$ is the
    maximal ideal of $V'$, where $\mathfrak m$ is the maximal ideal of $V$.
\end{itemize}
\end{lem}

In the proof of Theorem~\ref{thm.elemext} we use two preliminary reductions based on the
following lemmas. Let us call an
extension of valuation rings $V\subset V'$ \emph{good} if it satisfies the conclusion of
Theorem~\ref{thm.elemext}.

\begin{lem}\label{app.lem1}
If $V\subset V'\subset V''$ are extensions of valuation rings of characteristic $p$ with
$(V'')^p\subset V$ such that $V\subset V'$ and $V'\subset V''$ are good extensions, then also
$V\subset V''$ is a good extension.
\end{lem}

\begin{proof}
Let $\mathbf{ x} = (x_1,\ldots , x_r)$ be $p$-independent elements in the extension $V'/V$  and let
$\mathbf y = (y_1,\ldots , y_s)$ be $p$-independent elements in the extension $V''/V'$.
Then $\mathbf x , \mathbf y$ are $p$-independent in the extension $V''/V$, so
$V[\mathbf x , \mathbf y]$ is an 
elementary extension of $V$.

Consider a finitely generated $V$-subalgebra $A$ of $V''$; we have to show that for
suitable $\mathbf{ x}$ and $ \mathbf y$ as above we have
$A\subset V[\mathbf x , \mathbf y]$. Indeed, there exist $p$-independent elements
$\mathbf y$ in the extension $V''/V'$ such that $V'A\subset V'[\mathbf y ]$. So $A\subset
B[\mathbf y ]$ for a finitely generated $V$-subalgebra $B$ of $V'$. There exist $p$-independent elements
$\mathbf x$ in the extension $V'/V$ such that $B\subset V[\mathbf x ]$. Then we get
$A\subset V[\mathbf x, \mathbf y ]$ as requested.
\end{proof}

\begin{lem}\label{app.lem2}
Let $V\subset V'$ be an extension of valuation rings as above with $[K':K]=p$. Let
$\mathfrak p'\subset V'$ be a
prime ideal lying over a prime ideal $\mathfrak  p\subset V$. Assume that one of the following two
conditions holds:
\begin{itemize}
\item[(a)] the residue field extension $\kappa(\mathfrak  p')/\kappa(\mathfrak  p)$ is trivial and $V_{\mathfrak p}\subset
  V'_{\mathfrak p'}$ is good, or
\item[(b)] $[\kappa(\mathfrak p'):\kappa(\mathfrak p)]=p$ and $V/\mathfrak  p\subset
  V'/\mathfrak  p'$ is good.
\end{itemize}
Then the extension  $V\subset V'$ is  good. 
\end{lem}

\begin{proof}
In both parts of the proof we need the
\begin{claim}\label{app.claimval}
For $x\in
\mathfrak p'$ and $t\in V'\setminus \mathfrak p'$ we have $x/t\in V'$.
\end{claim}

\begin{proof}[Proof of Claim~\ref{app.claimval}]
If $x/t$ were not in $V'$, then $t/x$ would be in $V'$, but that would imply $t\in
\mathfrak p'$, which is a contradiction.
\end{proof}

  In order to prove Lemma~\ref{app.lem2} in case condition (a) holds, we start by observing that the assumptions imply that
\begin{equation}\label{app.eqfrak}
    V'_{\mathfrak p'}= \bigcup_{x\in \mathfrak p'} V_{\mathfrak p}[x]
  \end{equation}
  and that the system of elementary extensions of $ V_{\mathfrak p}$ in the union is
  directed.
  
  We prove that $V\subset V'$ is good  by showing that the elementary extensions of $V$ of the form $V[x/t]$ with $x\in
\mathfrak p'$ and $t\in V\setminus \mathfrak p$ are directed and that
their union is $V'$. Note that $x/t$ is automatically in $V'$ by Claim~\ref{app.claimval}.

For given  $x\in
\mathfrak p'$ we consider the ring
$R_x=\cup_{t\in V\setminus \mathfrak p} V[x/t]$. Then as $R_x=V+V_{\mathfrak p} x +
V_{\mathfrak p} x^2+\cdots$ we have $R_x=V_{\mathfrak p}[x] \cap V'$.
So by~\eqref{app.eqfrak} the system of rings $R_x$ where $x$ runs through $ \mathfrak p'$ is directed with union $V'$.

\smallskip

In order to prove Lemma~\ref{app.lem2} in case condition (b) holds, we show that there is a canonical map from the set $\bar {\mathcal S}$
 of elementary extensions $V/\mathfrak p \subsetneq \bar W \subset V'/\mathfrak p'$   to the set  $ {\mathcal S}$ of elementary
extensions $V\subsetneq W\subset V'$ given by
\[
\Phi\colon \bar{\mathcal S} \to  {\mathcal S}, \quad \Phi(\bar W ) = V'\times_{V'/\mathfrak p'}
\bar W.
\]
Once we show that $\Phi$ is well-defined and using that $V/\mathfrak  p\subset
  V'/\mathfrak  p'$ is good we immediately deduce that the set of elementary extensions
$\Phi(\bar{\mathcal S})$ is directed by inclusion and that $\cup_{\bar W\in\bar{\mathcal S} }
\Phi(\bar W)=V'$.

In order to show that $\Phi$ is well-defined consider an elementary extension
$ \bar W= V/\mathfrak p[\bar x]$ and lift $\bar x\in V'/\mathfrak p'$ to an element
$x\in V'$. Clearly, we have the inclusion $V[x]\subset \Phi (\bar W)$ and we claim that
equality holds. To see this start with an element $y\in \Phi (\bar W)\subset V'$. By
subtracting from $y$ a lift of $\bar y\in \bar W$ to $V[x]$ we can assume without loss of
generality that $y\in \mathfrak p'$. By Lemma~\ref{app:lembour} the ring extension
$V_{\mathfrak p}\subset V'_{\mathfrak p'}$ is finite and
$ V'_{\mathfrak p'}/\mathfrak p V'_{\mathfrak p'} = \kappa (\mathfrak p')$, so by Nakayama's lemma we
have $V'_{\mathfrak p'}=V_{\mathfrak p} [x]$ and we can write $y=a_0 + a_1 x + \cdots +
a_{p-1}x^{p-1}$ with $a_0,\ldots ,  a_{p-1}\in V_{\mathfrak p} $. As $y\in \mathfrak p'$
we actually have $a_0,\ldots ,  a_{p-1}\in \mathfrak p_{\mathfrak p} $. However, 
$\mathfrak p_{\mathfrak p}\subset V$ by Claim~\ref{app.claimval} and therefore $y\in V[x]$.
\end{proof}

\begin{proof}[Proof of Theorem~\ref{thm.elemext}] 
By Lemma~\ref{app.lem1} one reduces to $[K':K]=p$. By writing $K$ as a filtered colimit of
finitely generated fields we can also assume without loss of generality that the field extension $K/\mathbb F_p$ is finitely generated. Then the valuation is
of finite height; to see this, combine Proposition~VI.10.2.3 and Corollary~VI.10.3.1 from \cite{Bourbaki_comalg}. By induction on the height and using Lemma~\ref{app.lem2} one
reduces to the case of height one.

From now on we assume that the valuations are given by an absolute value $|\cdot | : K'\to
\mathbb R$ and that $[K':K]=p$.
We distinguish three cases:

\medskip

\noindent\textit{Case 1}: $[\kappa(V') : \kappa(V)]=p$. \\[.5ex]
In this case $V'$ is finite over $V$ and $V'/V'\fm= \kappa(V') $ by Lemma~\ref{app:lembour}, where $\fm$ is the
maximal ideal of $V$. Let $x\in V'$ be a lift of a generator of the field extension
$\kappa(V') / \kappa(V)$. Then $V'=V[x]$ by Nakayama's lemma.

\medskip

\noindent\textit{Case 2}: The valuation is discrete and $[|(K')^\times|:|K^\times |]=p$.\\[.5ex]
Choose a uniformizer $x\in V'$. Then for any element $y=a_0 + a_1 x + \cdots +  a_{p-1}
x^{p-1}$ in $K'$ with $a_0,\ldots, a_{p-1}\in K$ we have \[|y|=\max\{|a_0|,\ldots ,
|a_{p-1}x^{p-1}|\} \] as the non-zero real numbers in the max are pairwise different (in fact they
are pairwise different in $|(K')^\times|/|K^\times|\cong \mathbb Z /p\mathbb Z$). This means
that if $|y|\le 1$ then $a_0,\ldots , a_{p-1}\in V$, since $|V|=|K|\cap [0,|x|^{1-p}] $.

\medskip

\noindent\textit{Case 3}: Remaining cases.\\[.5ex]
Now Gabber's approximation method is applicable, which is explained in the next section.
For example, if $V$ is a discrete valuation ring and the ramification index is one, then
given a  sequence $(y_n)$  with the property of Proposition~\ref{prop.approx} one chooses 
$(w_n)$ such that $|w_n|=|x-y_n|$, which is possible
by $|K'|=|K|$. 

If the valuation is not discrete, then given a sequence
$(y_n)$ with the property of Proposition~\ref{prop.approx} it is possible to find a
sequence $(w_n)$ with the requested property by successively choosing $w_n\in V$ for $n\ge
1$ with \[ |x-y_n|\le |w_n|
\le \min\{ |w_{n-1}|, \frac{n+1}{n} |x-y_n|\} .\] 
This can be done since $|K^\times|$ is dense in $\mathbb R_{\ge 0}$.
\end{proof}

\subsection{Gabber's approximation method}
In this subsection let $K\subset K'$ be a purely inseparable extension of valued fields of
height one and of characteristic
$p$ with $[K':K]=p$. Let $V\subset V'$ be the corresponding extension of valuation rings.
Assume that $\kappa(V')=\kappa(V)$. Fix $x\in V'\setminus V$, so that $K'=K[x]$.

\begin{prop}[Gabber]\label{prop.approx}
Assume that there exist two sequences $(y_n)_n$ and $(w_n)_n$ in $V$ with the following
properties:
\begin{itemize}
\item $|x-y_n| $ is non-increasing, and for any $y\in V$ we have $|x-y_n|\le |x-y|$ for $n\gg
  0$,
  \item $|w_n|$ is non-increasing, $|x-y_n|\le |w_n| $ for all $n$, and \[\lim_{n\to \infty} |x-y_n|/|w_n|=1 .\]
\end{itemize}
Then
\[
V'=\bigcup_n V[\frac{x-y_n}{w_n}],
\]
and the system of subrings in the union is increasing in $n$.
\end{prop}

\begin{proof}
The fact that  the rings are increasing is easy and left to the reader. We show that any
element $v\in V'\setminus \{0 \}$ is contained
in $V[(x-y_n)/w_n]$ for $n\gg 0$. As $\kappa(V')=\kappa(V)$ we can assume without
loss of generality that $|v|<1$.
Set $z_n=x-y_n$ and write $v=a_0^{(n)} + \cdots +
a_{p-1}^{(n)}z_n^{p-1}$ with $a_0^{(n)},\ldots , a_{p-1}^{(n)}\in K $.

\begin{lemma} \label{lem.keyval}
  For $n\gg 0$ (depending on $v$) we have \[
    |v|=\max\{|a_0^{(n)}|, \ldots , |a_{p-1}^{(n)}z_n^{p-1} |   \} .
    \]
  \end{lemma}
  \begin{proof}
Observe that for a separable algebraic extension of valued fields $K\subset E$ we can without loss of
generality replace $K$ by $E$ and $K'$ by $E'=K'\otimes_K E$ in the proof of the lemma if the
element $x$ cannot be approximated closer by elements in $E$ than by elements in $K$.
Indeed, the
latter approximation property implies that the conditions of Proposition~\ref{prop.approx} also hold for the
extension $E'/E$.

\medskip

\noindent\textit{Step 1}: Replace $K$ by $K^h$ (henselization)\\[.5ex]
This is feasible because $K$ is dense in $K^h$. So given $y_E\in K^h$ find $y_K\in K$ with
$|y_E-y_K|<|y_E-x|$. Then $|x-y_E|=|x-y_K| $, so $x$ cannot be approximated closer in $K^h$
than in $K$.

\medskip

\noindent\textit{Step 2}: Replace $K$ by the splitting field $E$ of $ a_0^{(n)} + \cdots +
a_{p-1}^{(n)}X^{p-1}\in K[X]$\\[.5ex]
Note that the splitting field is independent of $n$ and that $d=[E:K]$ is prime to $p$.
Given $y_E\in E$ set $y_K=\tr (y_E)/d $. Then
\[
|x-y_K|= \left|\frac{1}{d} \sum_{\sigma \in \mathrm{Gal}(E/K)}  \sigma(x-y_E)  \right|\le |x-y_E|,
\]
so  $x$ cannot be approximated closer in $E$
than in $K$.
Here we used that $K$ is henselian which implies that $|\sigma z|=|z|$ for any $z\in E$
and $\sigma\in  \mathrm{Gal}(E/K)$.

\medskip

Now we can assume without loss of generality  that the polynomials $a_0^{(n)} + \cdots  + a_{p-1}^{(n)} X^{p-1}$ decompose into linear
factors over $K$. Then  Lemma~\ref{lem.keyval} is a consequence  of~\cite[Lemma~6.1.9]{GabberRamero}.
  \end{proof}

If we write
  \[
 v=a_0^{(n)} + (a_1 w_n) \frac{z_n}{w_n}  +\cdots +
(a_{p-1}^{(n)}w_n^{p-1}) (\frac{z_n}{w_n})^{p-1}
\]
the coefficients satisfy $|a_i^{(n)} w_n^i|\le (|w_n|/|z_n|)^i |v|$ for $n\gg 0$ by 
Lemma~\ref{lem.keyval}. But $$\lim_{n\to \infty}  (\frac{|w_n|}{|z_n|})^i|v| =|v|<1,$$
 so $v\in V[(x-y_n)/w_n]$ for $n\gg 0$.
\end{proof}

\bibliographystyle{amsalpha}
\bibliography{vorst}

\end{document}